\documentclass[11 pt]{article}

\usepackage[margin=1in]{geometry}
\usepackage{amsfonts ,amsthm, amsmath,amssymb,mathrsfs}
\usepackage{ragged2e}
\usepackage{cite}
\usepackage{blindtext}
\usepackage{changepage}
\usepackage{float}
\usepackage{multirow}
\usepackage{rotating}
\usepackage{lscape}
\renewcommand\neq{\mathrel{\vphantom{|}\mathpalette\xsneq\relax}}
\newcommand\xsneq[2]{%
  \ooalign{\hidewidth$#1|$\hidewidth\cr$#1=$\cr}%
}
\numberwithin{equation}{section}
\newtheorem{theorem}{Theorem}[section]
\newtheorem{lemma}{Lemma}[section]
\theoremstyle{plain}
\newtheorem{definition}[theorem]{Definition}
\newtheorem{example}{Example}[section]

\usepackage{graphicx}

\begin{document}
\title{\textbf{A numerical technique for solving multi-dimensional fractional optimal control problems using fractional wavelet method} }
\author{  S. Saha Ray, Akanksha Singh*\\
\textit{National Institute of Technology}\\
\textit{Department of Mathematics}\\
\textit{Rourkela-769008, India}\\
\textit{*19singhac@gmail.com}}

\maketitle

\begin{abstract} 
This paper presents an efficient numerical method for solving fractional optimal control problems using an operational matrix for a fractional wavelet. Using well-known formulae such as Caputo and Riemann-Liouville operators to determine fractional derivatives and integral fractional wavelets, operational matrices were devised and utilised to solve fractional optimal control problems. The proposed method reduced the fractional optimal control problems into a system of algebraic equations. To validate the effectiveness of the presented numerical approach, some illustrative problems were solved using fractional Taylor and Taylor wavelets, and the approximate cost function value derived by approximating state and control functions was compared. In addition, convergence rate and error bound of the proposed method have been derived.
\\
\\
\textbf{Keywords:}   Fractional optimal control problem,  fractional Taylor wavelet, Caputo derivative, Taylor wavelet,  Riemann–Liouville fractional integration, wavelet basis, operational matrix.   
\\
\textbf{Mathematics Subject Classification:}  49J15, 49N10, 65T60, 26A33.\\ 
\textbf{PACS Numbers:} 43.60.Hj, 02.30.Rz .\\

\end{abstract}
\numberwithin{equation}{section}
\section{Introduction}
Fractional calculus renders excellent models for real-life problems, and it has useful applications in many fields of science and engineering \cite{herrmann2013infrared, oldham1974fractional, sun2018new, machado2011recent, furati2021fractional,ray2015fractional,naik2020global}. Optimal control is the process of determining control and state functions for a dynamic system over a period of time to minimize or maximize a cost function. Despite the fact that optimal control theory has been studied for many years, fractional optimal control theory is a relatively new branch of mathematics. Various definitions of fractional derivatives can be applied to define a fractional optimal control problem (FOCP). Various numerical and analytical tools have been used in recent years to solve different kinds of FOCPs \cite{sahu2018comparison,agrawal1989general,agrawal2004general,
heydari2016wavelets,hosseinpour2019muntz,yousefi2011use,
dehestani2022spectral,barikbin2020solving,bhrawy2017solving}. Wavelets are extremely effective methods that are applied in a variety of numerical techniques. Most applications of the wavelet concept are found in applied science  and engineering. Wavelets have been found to be effective in a variety of applications, and they are particularly useful in signal processing \cite{behera2022efficient,saha2021new,behera2022wavelet,
behera2020operational,sahu2016legendre}
.

The operational matrix of fractional integration, the fractional Taylor wavelet basis, function approximation, and the Lagrange multiplier approach have been used in this study to solve a specific FOCP. It has been written as follows:

\begin{equation}
\begin{split}
\min\Tilde{{J}} & =  \dfrac{1}{2} \int_{{\zeta}_{0}}^{{\zeta}_{f}}(\Tilde{p}({\zeta})x^2({\zeta})+\Tilde{q}({\zeta})u^2({\zeta}))d{\zeta} \\
&= \int_{{\zeta}_{0}}^{{\zeta}_{f}} \mathfrak{F}({\zeta}, x({\zeta}), u({\zeta})) d {\zeta},
\end{split}\label{1.1}
\end{equation}
\begin{equation}
    {_{{\zeta}_{0}}^{C}}\mathcal{D}_{{\zeta}}^{\mu}x({\zeta})=\Tilde{a}({\zeta})x({\zeta})+\Tilde{b}({\zeta})u({\zeta}),\label{1.2}
\end{equation}
\begin{equation}
    x({\zeta}_{0})=x_{0}.\label{1.3}
\end{equation}
where $\Tilde{p}({\zeta})\geq{0}, \Tilde{q}({\zeta})>0, \Tilde{b}({\zeta})\neq{0}, {\zeta}_{0} = 0, {\zeta}_{f} =1 $.\\\\
With this proposed, the original optimal control problem has been converted into a set of linear equations. The fractional derivative of state function and control function have been expanded using the fractional Taylor wavelet basis . Then, the fractional integration and product operational matrices are used to obtain a linear system of algebraic equations from the cost function given in Eq. \eqref{1.1} and a dynamical system given by Eq. \eqref{1.2} in terms of the unknown coefficients. Finally, the constrained extremum technique is applied, which usually involves connecting the constraint equation derived from the given dynamical system to the cost function using a set of unknown Lagrange multipliers. Furthermore, the error estimation and convergence analysis of the proposed numerical technique have been established.
\\

The contents of this paper are as follows: Section 2 discusses the definitions of fractional derivatives and integrals \cite{li2011riemann,caputo1967linear}. Wavelets  are discussed in Section 3. Section 4 describes function approximation. Section 5 illustrates the operational matrix of integration for the fractional Taylor wavelet\cite{keshavarz2018taylor}. In Section 6, a numerical approach has been proposed for solving the FOCP. Section 7 explains the estimation of the error. In Section 8, the convergence of the proposed fractional Taylor wavelets method has been established . In Section 9, numerical experimental results show the accuracy and efficiency of the proposed numerical scheme. Finally, Section 10 comes to an end with a concluding remark.

\section{ Preliminaries }

In this section, some important definitions are provided.
\begin{definition}[Riemann-Liouville integral]
The Riemann-Liouville fractional integral of a function $\Tilde{\mathcal{Z}}( \sigma)$ of order $\mu > 0$ is defined  \cite{li2011riemann} as follows:
\begin{equation}
 {_0}{}\mathcal{I}_{{\sigma}}^{\mu} \Tilde{\mathcal{Z}}( \sigma) =\dfrac{1}{\Gamma(\mu)}  \int_{0}^{\sigma} (\sigma-\varpi)^{\mu-1}\Tilde{\mathcal{Z}}(\varpi)d\varpi, \ \ \ \ \     \mu> 0, \sigma > 0. \label{2.1}
\end{equation}

\end{definition}

\begin{definition}[Caputo fractional derivative ]
Caputo fractional derivative was introduced by M. Caputo in 1967. The Caputo fractional derivative of a function $\Tilde{\mathcal{Z}}({\sigma})$ of order $\mu$ is defined \cite{caputo1967linear} as follows:

\begin{equation}
{_{\sigma_{0}}^C}\mathcal{D}_{\sigma}^{\mu}\Tilde{\mathcal{Z}}(\sigma)=
\begin{cases} 
\dfrac{1}{\Gamma(1-\mu)}\displaystyle\int_{\sigma_{0}}^{\sigma}{(\sigma-\varpi)^{-\mu}}\frac{d\Tilde{\mathcal{Z}}(\varpi)}{d\varpi}d\varpi, \ \ & 0<\mu<1, \\  
\mathcal{Z}'(\sigma), & \mu=1.
\end{cases}\label{2.2} 
\end{equation}
\end{definition}

\section{Wavelets}
This section gives an overview of Taylor wavelets and fractional Taylor wavelets.
\subsection{Taylor wavelets}
Let ${{k}} $ be a positive integer. For each ${{n}} =1,2,  \dots  , 2^{{{k}}-1}$ and non-negative integer $ {{m}}=0,1,2, \dots,{M}-1$ where $ {m} $ is the order of Taylor polynomial, then  the Taylor wavelet $\psi_{{{n}}, {{m}}}({\zeta})$ is defined on the interval $[0, 1)$ as 
\begin{equation}
\psi_{{{n}}, {{m}}}(\varsigma)=
\begin{cases}
2^{\frac{{{k}}-1}{2}}\Tilde{\mathcal{T}}_{{{m}}}(2^{{{k}}-1}\varsigma-{{n}}+1),  \quad if \quad \dfrac{{{n}}-1}{2^{{{k}}-1}}\le \varsigma <\dfrac{{{n}}}{2^{{{k}}-1}},\\
0, \quad \quad \quad \quad \quad \quad \quad \quad \quad \quad \quad   \quad otherwise,\\
\end{cases}\label{3.1}
\end{equation}
where $ \Tilde{\mathcal{T}}_{{{m}}}(\varsigma)=\sqrt{2{{m}}+1}\times \varsigma^{{{m}}}$ is the normalized Taylor polynomial of degree ${{m}}$.
\subsection{Fractional order Taylor wavelets}
By applying the transformation $\varsigma={\zeta}^\mu$ to Eq. \eqref{3.1} for a positive real number $\mu$, the fractional-order Taylor wavelet is defined \cite{yuttanan2021numerical} as follows:
\begin{equation}
\psi_{{{n}}, {{m}}}^{\mu} ({\zeta})=
\begin{cases}
2^{\frac{{{k}}-1}{2}}\Tilde{\mathcal{T}}_{{{m}}}(2^{{{k}}-1}{\zeta}^{\mu}-{{n}}+1),  \quad if \quad \Big(\dfrac{{{n}}-1}{2^{{{k}}-1}}\Big)^\frac{1}{\mu}\le {\zeta} <\Big(\dfrac{{{n}}}{2^{{{k}}-1}}\Big)^\frac{1}{\mu},\\
0, \quad \quad \quad \quad \quad \quad \quad \quad \quad \quad \quad \quad \quad otherwise.\\
\end{cases}\label{3.2}
\end{equation}
\section{Function approximation}
A function $f({\zeta})$, square integrable in $[0,1]$,  can be expressed in terms of the wavelets as
\begin{equation}
f({\zeta})\approx\sum_{{{n}}=1}^{2^{{{k}}-1}}\sum_{{{m}}=0}^{{M}-1}\Hat{c}_{{{n}}, {{m}}}\psi_{{{n}}, {{m}}}^{\mu}({\zeta})=\Hat{C}^{{T}}\Hat{\Psi}^{\mu}({\zeta}), \label{4.1}
\end{equation}
where $\Hat{C}$ and $\Hat{\Psi}^{\mu}({\zeta})$ are column vectors of dimension $(2^{{{k}}-1}{M}\times1)$ given by
\begin{equation}
\Hat{C}^{T}=[\Hat{c}_{1, 0}, \Hat{c}_{1, 1}, \Hat{c}_{1, 2},  \dots \Hat{c}_{1, {M}-1},  \dots , \Hat{c}_{2^{{{k}}-1}, {M}-1}], \label{4.2}
\end{equation}
and
\begin{align}\nonumber
 \Hat{\Psi}^{\mu}({\zeta}) &=[\psi_{1, 0}^{\mu}({\zeta}), \psi_{1, 1}^{\mu}({\zeta}), \psi_{1, 2}^{\mu}({\zeta}), \dots , \psi_{1, {M}-1}^{\mu}({\zeta}), \dots , \psi_{2^{{{k}}-1}, {M}-1}^{\mu}({\zeta})]^T \\
&=[\psi_{1}^{\mu}({\zeta}), \psi_{ 2}^{\mu}({\zeta}), \psi_{3}^{\mu}({\zeta}), \dots ,  \psi_{\Hat{{m}}}^{\mu}({\zeta})]^T,\label{4.3}
\end{align} 
where $\Hat{{m}}=2^{{{k}}-1}{M}.$\\
The coefficient vector $\Hat{C}$ can be obtained from Eq. \eqref{4.1} as follows
\begin{equation}
\Hat{C}^{T}=\Bigg(\int_0^1f({\zeta})\left(\Hat{\Psi}^{\mu}({\zeta})\right)^Td{\zeta}\Bigg) \Hat{D}^{-1}(\mu),\label{4.4}
\end{equation}
where $\Hat{D}(\mu)$ is a square matrix of order $2^{{{k}}-1}{M}$ and is given by
\begin{equation}
\Hat{D}(\mu)=\int_0^1 \Hat{\Psi}^{\mu}({\zeta})(\Hat{\Psi}^{\mu}({\zeta}))^Td{\zeta}. \label{4.5}
\end{equation}

When ${{k}} = 2$ , ${M} = 4$, fractional Taylor wavelets for $\mu=0.9$ 
$$
\Hat{D}(0.9)=
\begin{bmatrix}
0.925875 & 0.844033 &	0.7394 & 0.662063 & 0 & 0 & 0 & 0\\
0.844033 &	0.992009 &	0.969161 & 0.922368 & 0 & 0 & 0 & 0\\
0.7394 &	0.969161 &	1.00639	& 0.995918 &  0 & 0 & 0 & 0\\
0.662063 &	0.922368 &	0.995918 &	1.01268	 &  0 & 0 & 0 & 0\\
 0 & 0 & 0 & 0 &	1.07413 & 0.941951 &	0.815443 &	0.72606\\
 0 & 0 & 0 & 0 &	0.941951 &	1.09403	& 1.06284 &	1.00824\\
 0 & 0 & 0 & 0 &	0.815443 &	1.06284 &	1.10008	 & 1.08633\\
 0 & 0 & 0 & 0 & 0.72606	& 1.00824 &	1.08633 &	1.10297\\
\end{bmatrix}.
\quad
$$
\\
When $\mu=1 $,  fractional Taylor wavelets change into Taylor wavelets
$$
\Hat{D}(1)=
\begin{bmatrix}
1	& 0.866025 &	0.745356 &	0.661438  &	 0 & 0 & 0 & 0\\
0.866025 &	1 &	0.968246 &	0.916515  &	 0 & 0 & 0 & 0\\
0.745356 &	0.968246 &	1	&0.986013 &	 0 & 0 & 0 & 0\\
0.661438 &	0.916515 &	0.986013 &	1 &	0. &	0.	& 0. &	0.\\
 0 & 0 & 0 & 0 &1	& 0.866025 &	0.745356 &	0.6614380\\
 0 & 0 & 0 & 0 	&.0.866025 &	1. &	0.968246 &	0.916515\\
 0 & 0 & 0 & 0 	&	0.745356 &	0.968246 &	1	&0.986013\\
 0 & 0 & 0 & 0 	& 0.661438 &	0.916515 &	0.986013 &	1\\

\\
\end{bmatrix}.
\quad
$$
\section{Fractional Taylor wavelet operational matrix}
In the present analysis, the operational matrix for first-order integral has been derived. 
\begin{theorem}
Let $ \Hat{\Psi}^{\mu}({\zeta})$ be the fractional Taylor wavelets vector introduce in Eq. \eqref{4.3}. Then the first order integral operational matrix is given by
$$\int_{0}^{{\zeta}}\Hat{\Psi}^{\mu}({\zeta})d{\zeta}\approx \mathcal{P}^1\Hat{\Psi}^{\mu}({\zeta}).$$
Also, the n,m th element of first order integral of vector $\Hat{\Psi}^{\mu}({\zeta})$ is given by
$$\int_{0}^{{\zeta}} \psi_{{n},{m}}^{\mu}({\zeta}) d{\zeta} \approx \sum_{r=1}^{2^{{{k}}-1}}\sum_{l=0}^{{M}-1}\Theta_{r l}^{\mu; {n} {{m}}}\psi_{{r}, {l}}^{\mu}({\zeta}), ~ ~ ~ ~ ~ ~ ~ ~ ~n=1,2,\dots,2^{{k}-1},\quad{{m}} = 0, 1, 2,  \dots , {M}-1$$
where
\begin{equation*}
\Theta_{r l}^{\mu; {n} {{m}}}=
\begin{cases}
 2^{\frac{{{k}}-1}{2}}\sqrt{2{{m}}+1} (2^{{{k}}-1})^{{{m}}} c_{rl}^{1{{m}}}, \quad  \quad \quad \quad \quad \quad \quad \quad \quad \quad \quad \quad \quad n=1,\\
 2^{\frac{{{k}}-1}{2}} \sqrt{2{{m}}+1}\displaystyle{\sum_{s=0}^{{{m}}}}2^{({{k}}-1)s}(-1)^{{{m}}-s}({{n}}-1)^{{{m}}-s}\binom{{{m}}}{s}c_{{r}{l}}^{{{n}}s}, \quad \quad n\neq 1.
\end{cases}
\end{equation*}
\end{theorem} 
\begin{proof}
The fractional  integration of the fractional Taylor wavelets vector $\Hat{\Psi}^{\mu}({\zeta})$ can be determined as
\begin{equation}
\mathcal{I}^\mu\Hat{\Psi}^{\mu}({\zeta})\approx \mathcal{P}^\mu\Hat{\Psi}^{\mu}({\zeta}),\label{5.1}
\end{equation}
where $\mathcal{P}^\mu$ is operational matrix of order $2^{{{k}}-1}{M}$.\\
Whence the first-order integration of the fractional Taylor wavelets vector $\Hat{\Psi}^{\mu}({\zeta})$ can be determined as
\begin{equation}
\int_{0}^{{\zeta}}\Hat{\Psi}^{\mu}({\zeta})d{\zeta} \approx \mathcal{P}^{1}\Hat{\Psi}^{\mu}({\zeta}).\label{5.2}
\end{equation}
Fractional Taylor wavelets can be written as follows:
\begin{align}
\psi_{{{n}}, {{m}}}^{\mu} ({\zeta}) = 2^{\frac{{{k}}-1}{2}}\Tilde{\mathcal{T}}_{{{m}}}\left(2^{{{k}}-1}{\zeta}^{\mu}-{{n}}+1\right)\chi({\zeta})\big\vert_{{\zeta}\in \left[\left(\frac{{{n}}-1}{2^{{{k}}-1}}\right)^\frac{1}{\mu}, \left(\frac{{{n}}}{2^{{{k}}-1}}\right)^\frac{1}{\mu}\right)},\label{5.3}
\end{align}  
where ~ $\chi({\zeta})\big\vert_{{\zeta}\in \left[\left(\frac{{{n}}-1}{2^{{{k}}-1}}\right)^\frac{1}{\mu}, \left(\frac{{{n}}}{2^{{{k}}-1}}\right)^\frac{1}{\mu}\right)}$  is characteristic  function defined as follows:
\begin{equation}
\chi({\zeta})\big\vert_{{\zeta}\in \left[\left(\frac{{{n}}-1}{2^{{{k}}-1}}\right)^\frac{1}{\mu}, \left(\frac{{{n}}}{2^{{{k}}-1}}\right)^\frac{1}{\mu}\right)}=
\begin{cases}
1, ~ ~ ~ ~ ~ ~{\zeta}\in \left[\left(\frac{{{n}}-1}{2^{{{k}}-1}}\right)^\frac{1}{\mu}, \left(\frac{{{n}}}{2^{{{k}}-1}}\right)^\frac{1}{\mu}\right),\\
0, ~ ~ ~ ~ ~ ~ ~ ~ ~ ~ otherwise.\label{5.4}
\end{cases}  
\end{equation}
Then
\begin{align}
\psi_{{{n}}, {{m}}}^{\mu} ({\zeta})=2^{\frac{{{k}}-1}{2}}\sqrt{2{{m}}+1}\left(2^{{{k}}-1}{\zeta}^{\mu}-{{n}}+1\right)^{{{m}}}\chi({\zeta})\big\vert_{{\zeta}\in \left[\left(\frac{{{n}}-1}{2^{{{k}}-1}}\right)^\frac{1}{\mu}, \left(\frac{{{n}}}{2^{{{k}}-1}}\right)^\frac{1}{\mu}\right)},\label{5.5}
\end{align}
\begin{align}\nonumber
\psi_{{{n}}, {{m}}}^{\mu} ({\zeta}) = & 2^{\frac{{{k}}-1}{2}} \sqrt{2{{m}}+1}\sum_{s=0}^{{{m}}}\binom{{{m}}}{s}\left(2^{{{k}}-1}{\zeta}^{\mu}\right)^s(-1)^{{{m}}-s}({{n}}-1)^{{{m}}-s}\\
&\times \chi({\zeta})\big\vert_{{\zeta}\in \left[\left(\frac{{{n}}-1}{2^{{{k}}-1}}\right)^\frac{1}{\mu}, \left(\frac{{{n}}}{2^{{{k}}-1}}\right)^\frac{1}{\mu}\right)},\label{5.6}
\end{align}
\begin{align}
\mathcal{I}\left(\psi_{{n}, {{m}}}^{\mu} ({\zeta})\right)=2^{\frac{{{k}}-1}{2}} \mathcal{I}\left(\Tilde{\mathcal{T}}_{{{m}}}\left(2^{{{k}}-1}{\zeta}^{\mu}-{{n}}+1\right)\chi({\zeta})\big\vert_{{\zeta}\in \left[\left(\frac{{{n}}-1}{2^{{{k}}-1}}\right)^\frac{1}{\mu}, \left(\frac{{{n}}}{2^{{{k}}-1}}\right)^\frac{1}{\mu}\right)}\right).\label{5.7}
\end{align}
When ${{n}}=1$
\begin{align}\nonumber
\mathcal{I}\left(\psi_{1, {{m}}}^{\mu} ({\zeta})\right) & =2^{\frac{{{k}}-1}{2}} \mathcal{I}\left(\Tilde{\mathcal{T}}_{{{m}}}\left(2^{{{k}}-1}{\zeta}^{\mu}\right)\chi({\zeta})\big\vert_{{\zeta}\in \left[\left(\frac{{{n}}-1}{2^{{{k}}-1}}\right)^\frac{1}{\mu}, \left(\frac{{{n}}}{2^{{{k}}-1}}\right)^\frac{1}{\mu}\right)}\right)\\  
&=2^{\frac{{{k}}-1}{2}}\sqrt{2{{m}}+1}\left(2^{{{k}}-1}\right)^{{{m}}} \mathcal{I}
\left({\zeta}^{\mu {{m}}}\chi({\zeta})\big\vert_{{\zeta}\in \left[\left(\frac{{{n}}-1}{2^{{{k}}-1}}\right)^\frac{1}{\mu}, \left(\frac{{{n}}}{2^{{{k}}-1}}\right)^\frac{1}{\mu}\right)}\right).\label{5.8}
\end{align}
When ${{n}} = 2, 3, 4,  \dots  ,2^{{{k}}-1}$, 
Eq. \eqref{5.7} can be written  by  using Eq. \eqref{5.6} as follows   
\begin{align}\nonumber
\mathcal{I}\left(\psi_{{n} , {{m}}}^{\mu} ({\zeta})\right)= &2^{\frac{{{k}}-1}{2}} \sqrt{2{{m}}+1}\sum_{s=0}^{{{m}}}2^{({{k}}-1)s}(-1)^{{{m}}-s}({{n}}-1)^{{{m}}-s}\binom{{{m}}}{s}\\
& \times \mathcal{I}\left({\zeta}^{\mu s}\chi({\zeta})\big\vert_{{\zeta}\in \left[\left(\frac{{{n}}-1}{2^{{{k}}-1}}\right)^\frac{1}{\mu}, \left(\frac{{{n}}}{2^{{{k}}-1}}\right)^\frac{1}{\mu}\right)}\right).\label{5.9}
\end{align}
Now, approximating $\mathcal{I}\left({\zeta}^{\mu s}\chi({\zeta})\big\vert_{{\zeta}\in \left[\left(\frac{{{n}}-1}{2^{{{k}}-1}}\right)^\frac{1}{\mu}, \left(\frac{{{n}}}{2^{{{k}}-1}}\right)^\frac{1}{\mu}\right)}\right)$ by $2^{{{k}}-1}{M}$ terms of fractional Taylor wavelets
\begin{equation}
\mathcal{I}\left({\zeta}^{\mu s}\chi({\zeta})\big\vert_{{\zeta}\in \left[\left(\frac{{{n}}-1}{2^{{{k}}-1}}\right)^\frac{1}{\mu}, \left(\frac{{{n}}}{2^{{{k}}-1}}\right)^\frac{1}{\mu}\right)}\right)=g_{{{n}}s}({\zeta})\approx \sum_{r=1}^{2^{{{k}}-1}}\sum_{l=0}^{{M}-1}c_{rl}^{{n} s}\psi_{r,l}^{\mu}({\zeta})=C_{rl}^{T}\Hat{\Psi}^{\mu}({\zeta}) ,  \label{5.10}
\end{equation}
where
\begin{align}
C_{rl}^{T} =<g_{{{n}}s}({\zeta}), \Hat{\Psi}^{\mu}({\zeta})>\Hat{D}^{-1}(\mu).\label{5.11}
\end{align}
\\
Substituting Eq. \eqref{5.10} into Eqs. \eqref{5.8} and \eqref{5.9}, this yields 
\begin{align}\nonumber
\mathcal{I}\left(\psi_{1,{{m}}}^{\mu} ({\zeta})\right) & \approx 2^{\frac{{{k}}-1}{2}}\sqrt{2{{m}}+1} \left(2^{{{k}}-1}\right  )^{{{m}}}\sum_{r=1}^{2^{{{k}}-1}}\sum_{l=0}^{{M}-1}c_{rl}^{1{{m}}}\psi_{r,l}^{\mu}({\zeta})\\
&=\sum_{r=1}^{2^{{{k}}-1}}\sum_{l=0}^{{M}-1}\Theta_{r l}^{\mu; 1 {{m}}}\psi_{{r}, {l}}^{\mu}({\zeta}), ~ ~ ~ ~ ~ ~ ~ ~ ~{{m}} = 0, 1, 2,  \dots , {M}-1,\label{5.12}
\end{align}
where
$\Theta_{r l}^{\mu; 1 {{m}}}= 2^{\frac{{{k}}-1}{2}}\sqrt{2{{m}}+1} \left(2^{{{k}}-1}\right)^{{{m}}} c_{rl}^{1{{m}}}$.\\
\begin{align}\nonumber
\mathcal{I}\left(\psi_{{{n}}, {{m}}}^{\mu} ({\zeta})\right) & \approx 2^{\frac{{{k}}-1}{2}} \sqrt{2{{m}}+1} \sum_{s=0}^{{{m}}}2^{({{k}}-1)s}(-1)^{{{m}}-s}({{n}}-1)^{{{m}}-s}\binom{{m}}{s}\sum_{r=1}^{2^{{{k}}-1}}\sum_{l=0}^{{M}-1}c_{rl}^{{{n}}s}\psi_{r,l}^{\mu}({\zeta})\\
&=\sum_{r=1}^{2^{{{k}}-1}}\sum_{l=0}^{{M}-1}\Theta_{r l}^{\mu; {{n}} {{m}}}\psi_{{r}, {l}}^{\mu}({\zeta})\quad {{n}}=2,3, \dots ,2^{{{k}}-1} ~~,
\quad {{m}} = 0, 1, 2,  \dots ,{M}-1,  \label{5.13}    
\end{align}
where~~$\Theta_{rl}^{\mu; {{n}}{{m}}}= 2^{\frac{{{k}}-1}{2}} \sqrt{2{{m}}+1}\displaystyle{\sum_{s=0}^{{{m}}}}2^{({{k}}-1)s}(-1)^{{{m}}-s}({{n}}-1)^{{{m}}-s}\binom{{{m}}}{s}c_{{r}{l}}^{{{n}}s} $.
\\  
\begin{align*}
\int_{0}^{{\zeta}} \psi_{{n},{m}}^{\mu}({\zeta}) d{\zeta} &\approx \sum_{r=1}^{2^{{{k}}-1}}\sum_{l=0}^{{M}-1}\Theta_{r l}^{\mu; {n} {{m}}}\psi_{{r}, {l}}^{\mu}({\zeta}) ~ ~ ~ ~ ~ ~ ~ ~ ~n=1,2,\dots,2^{{k}-1},\quad{{m}} = 0, 1, 2,  \dots , {M}-1\\
&= \mathcal{P}^{1}\Hat{\Psi}^{\mu}({\zeta}),
\end{align*}
where
$$ \mathcal{P}^{1} = 
\begin{bmatrix}  
\Theta_{1 0}^{\mu; 1 0} &  \dotsb  &\Theta_{1 {M}-1}^{\mu; 1 0} & \dotsb& \Theta_{2^{{{k}}-1} 0}^{\mu; 1 0}& \dotsb     &     \Theta_{2^{{{k}}-1} {M}-1}^{\mu; 1 0}\\
 \Theta_{1 0}^{\mu; 1 1} & \dotsb &\Theta_{1 {M}-1}^{\mu; 1 1} & \dotsb & \Theta_{2^{{{k}}-1}0}^{\mu; 1 1}& \dotsb    &     \Theta_{2^{{{k}}-1} {M}-1}^{\mu; 1 1}\\
  &  &         &   &  &  \\ \\
\vdots  & \dotsb  &   \vdots       & \dotsb   & \vdots  & \ddots & \vdots\\
  &  &         &   &  &  \\ \\
\Theta_{1 0}^{\mu; 2^{{{k}}-1} {M}-1} & \dotsb &\Theta_{1 {M}-1}^{\mu; 2^{{{k}}-1} {M}-1} &\dotsb & \Theta_{2^{{{k}}-1} 0}^{\mu; 2^{{{k}}-1} {M}-1}& \dotsb     &     \Theta_{2^{{{k}}-1} {M}-1}^{\mu; 2^{{{k}}-1} {M}-1}\\ 
\quad
\end{bmatrix}.
\quad
$$  
\end{proof}
\begin{landscape}
When $ {{k}} = 2, {M} = 4$ and $\mu=0.9$, operational matrix for fractional Taylor wavelet
\\
\\
\\
\\
$ 
\mathcal{P}^{0.9}=
\begin{bmatrix}
0.000376754   & 0.297652     &	0.00288863    & -0.00145978	 & 0.514575	  & -0.0901146   &	0.0680759	  & -0.0250085 \\
-0.0000339738 &	0.000166643  &	0.220955	  & 0.000139662  &	0.488206	  & -0.112764	 & 0.0940739	  & -0.0357926 \\
4.32235\times10^{-6} &	-0.000021994 &	0.000036755   &	0.168787     & 0.438811	  & -0.118706    &	0.104177      &	-0.0403629\\
-0.00681141   &	0.0723549    &	-0.24335      &	0.313394     &	0.400278      &	-0.120455    &	0.109204 7	  &-0.0428148\\
 0 & 0 & 0 & 0        &	0.00523984    & 0.389532     &	-0.0681946    &	0.025082 \\
 0 & 0 & 0 & 0    &  -0.000555412  & 0.0104491    &	0.244884      &	-0.00747049 \\
 0 & 0 & 0 & 0         &
0.000177077   &	-0.00257342  &	0.015652      &	0.172614 \\
 0 & 0 & 0 & 0       &	 -0.00604376   &	0.0694403    &	-0.242547     &	0.327425 \\
\end{bmatrix}
\quad
$ 
\\
\\
\\
\\
\\
\\
\\
\\
For Taylor wavelet
\\
\\
$
\mathcal{P}^{0.9}=
\begin{bmatrix}

0.0048894 & 0.381098 & -0.080508 &	0.0277208 &0.552325 &	-0.091867 &	0.070449	 & -0.0261748 \\
-0.000615 & 0.011247 &	0.235221	 & -0.0140564 &	0.500057 &	-0.113807 &	0.0971996 &	-0.0375117 \\
0.0003976 &	-0.005255 &	0.0257836 &	0.152538	 & 0.442586 &	-0.119822 & 	0.107959	  & -0.0424765 \\
-0.005181 & 0.0603413	& -0.214021 &	0.297002 &	0.400694 &	-0.121823 &	0.113584	 & -0.045259 \\
 0 & 0 & 0 & 0	 &	0.0048894	& 0.381098 & 	-0.080508	& 0.0277208 \\
 0 & 0 & 0 & 0 & -0.000615	 & 0.011247	& 0.235221	& -0.0140564 \\
 0 & 0 & 0 & 0 &	0.0003976 &	-0.0052557	& 0.0257836	& 0.152538 \\
 0 & 0 & 0 & 0& 	-0.005181	& 0.0603413	& -0.214021	& 0.297002 \\
\end{bmatrix}
\quad
$
\end{landscape}
\section{Numerical scheme }

Consider the problem given by Eqs. \eqref{1.1}-\eqref{1.3} and approximate the undetermined functions $x({\zeta})$  and $u({\zeta})$ by fractional Taylor wavelets functions 

\begin{equation}
   \min\Tilde{{J}}  = \dfrac{1}{2} \int_{0}^{1}(\Tilde{p}({\zeta})x^2({\zeta})+\Tilde{q}({\zeta})u^2({\zeta}))d{\zeta} ,\label{6.1}
\end{equation}

\begin{equation}
    {_{0}^C}\mathcal{D}_{{\zeta}}^{\mu}x({\zeta})= \Tilde{a}({\zeta})x({\zeta})+\Tilde{b}({\zeta})u({\zeta}),\label{6.2}
\end{equation}

\begin{equation}
    x(0)= x_{0}.\label{6.3}
\end{equation}
\\
First, we expand the fractional derivative of the state function and control function by the fractional Taylor wavelets basis $\Hat{\Psi}^{\mu}({\zeta})$.
\begin{equation}
   {_{0}^C}\mathcal{D}_{{\zeta}}^{\mu}x({\zeta})\approx \Hat{C}^{T}\Hat{\Psi}^{\mu}({\zeta}),\label{6.4}
\end{equation}
\begin{equation}
      u({\zeta}) \approx \Hat{U}^{T}\Hat{\Psi}^{\mu}({\zeta}),\label{6.5}
\end{equation}
where $\Hat{C}^{T}=[\Hat{c}_{1}, \Hat{c}_{2}, \dots , \Hat{c}_{\Hat{{m}}}],$ $
\Hat{U}^{T}=[\Hat{u}_{1}, \Hat{u}_{2}, \dots , \Hat{u}_{\Hat{{m}}}]$
are unknowns. \\

Also approximating functions $ \Tilde{a}({\zeta}), \Tilde{b}({\zeta}), \Tilde{p}({\zeta}), \Tilde{q}({\zeta}), x(0) $ by the fractional Taylor wavelets basis as
\begin{align}
     x(0)\approx d_{1}^{T}\Hat{\Psi}^{\mu}({\zeta}), ~~~~~
      {_0}{}\mathcal{I}_{{\zeta}}^{\mu}\Hat{\Psi}^{\mu}({\zeta}) \approx \mathcal{P}^{\mu}\Hat{\Psi}^{\mu}({\zeta}),\label{6.6}
\end{align}
where 
\begin{align*}
\mathcal{P}^{\mu}=\Big(\int_{0}^{1}( {_0}{}\mathcal{I}_{{\zeta}}^{\mu}\Hat{\Psi}^{\mu}({\zeta})) (\Hat{\Psi}^{\mu}({\zeta}))^{T}d{\zeta}\Big)\Hat{D}^{-1}(\mu),
\end{align*}
\begin{align}
\Tilde{a}({\zeta})\approx \Hat{A}^{T}\Hat{\Psi}^{\mu}({\zeta}),~~~~~  \Tilde{b}({\zeta})\approx \Hat{B}^{T}\Hat{\Psi}^{\mu}({\zeta}),\label{6.7}
\end{align}
\begin{align}
\Tilde{p}({\zeta})\approx \Hat{P}^{T}\Hat{\Psi}^{\mu}({\zeta}),~~~~~ \Tilde{q}({\zeta})\approx \Hat{Q}^{T}\Hat{\Psi}^{\mu}({\zeta}),\label{6.8}
\end{align}
where
\begin{equation*}
\Hat{A}^{T}=[ \Hat{a}_{1},\Hat{a}_{2},  \dots , \Hat{a}_{\Hat{{m}}}], ~~~~
\Hat{B}^{T}=[ \Hat{b}_{1},\Hat{b}_{2},  \dots , \Hat{b}_{\Hat{{m}}}],
\end{equation*}
\begin{equation*}
\Hat{P}^{T}=[ \Hat{p}_{1},\Hat{p}_{2},  \dots , \Hat{p}_{\Hat{{m}}}], ~~~~
\Hat{Q}^{T}=[ \Hat{q}_{1},\Hat{q}_{2},  \dots , \Hat{q}_{\Hat{{m}}}],
\end{equation*}

\begin{equation*}
\Hat{a}_{l}=\Big(\int_0^1 \Tilde{a}({\zeta})(\Hat{\Psi}^{\mu}({\zeta}))^{T}d{\zeta}\Big)\Hat{D}^{-1}(\mu),~~~~
 \Hat{b}_{l}=\Big(\int_0^1 \Tilde{b}({\zeta})(\Hat{\Psi}^{\mu}({\zeta}))^{T}d{\zeta}\Big)\Hat{D}^{-1}(\mu),
\end{equation*}
\begin{equation*}
\Hat{p} _{l}=\Big(\int_0^1 \Tilde{p}({\zeta})(\Hat{\Psi}^{\mu}({\zeta}))^{T}d{\zeta}\Big)\Hat{D}^{-1}(\mu),~~~~ \Hat{q}_{l}=\Big(\int_0^1 \Tilde{q}({\zeta})(\Hat{\Psi}^{\mu}({\zeta}))^{T}d{\zeta}\Big)\Hat{D}^{-1}(\mu),
\end{equation*}
for $l=1,2,\dots,\Hat{{m}}.$\\\\
Using  Riemann–Liouville fractional operational matrix of integration, $x({\zeta})$ can be represented as
\begin{equation}
   {_0}{}\mathcal{I}_{{\zeta}}^{\mu}{_{0}^C}{}\mathcal{D}_{{\zeta}}^{\mu}x({\zeta})=x({\zeta})-x(0),\label{6.9}
\end{equation}
\begin{equation}
x({\zeta}) =  {_0}{}\mathcal{I}_{{\zeta}}^{\mu}{_{0}^C}{}\mathcal{D}_{{\zeta}}^{\mu}x({\zeta}) + x(0) \approx (\Hat{C}^{T} \mathcal{P}^{\mu} + d_{1}^{T}) \Hat{\Psi}^{\mu}({\zeta}),\label{6.10}
\end{equation}
where $\mathcal{P}^\mu$ is the fractional Taylor operational matrix .\\\\
Now, approximating the $x^2(t)$ by using fractional Taylor wavelets 
\begin{align}\nonumber
x^2({\zeta}) &\approx((\Hat{C}^{T}\mathcal{P}^{\mu}+d_{1}^{T})\Hat{\Psi}^{\mu}({\zeta}))^2\\ \nonumber
&=((\Hat{C}^{T}P^{\mu}+d_{1}^{T})\Hat{\Psi}^{\mu}({\zeta}))((\Hat{C}^{T}\mathcal{P}^{\mu}+d_{1}^{T})\Hat{\Psi}^{\mu}({\zeta}))^T\\ \nonumber
&=(\Hat{C}^{T}P^{\mu}+d_{1}^{T})\Hat{\Psi}^{\mu}({\zeta})(\Hat{\Psi}^{\mu}({\zeta}))^{T}(\Hat{C}^{T}\mathcal{P}^{\mu}+d_{1}^{T})^{T}\\
&=C_{2}^{T}\Hat{\Psi}^{\mu}({\zeta})(\Hat{\Psi}^{\mu}({\zeta}))^{T}C_{2}, \label{6.11} 
\end{align}
where $ C_{2}=(\Hat{C}^{T}\mathcal{P}^{\mu}+d_{1}^{T})^{T}$.\\
\begin{equation}
\Hat{\Psi}^{\mu}({\zeta})(\Hat{\Psi}^{\mu}({\zeta}))^{T}C_{2}\approx\tilde{C}\Hat{\Psi}^{\mu}({\zeta}),\label{6.12}
\end{equation}
where $\tilde{C}$ is product operational matrix.
\begin{equation}
  ~~\tilde{C}=\Big(\int_{0}^{1} (\Hat{\Psi}^{\mu}({\zeta})(\Hat{\Psi}^{\mu}({\zeta}))^TC_{2}(\Hat{\Psi}^{\mu}({\zeta}))^{T})d{\zeta}\Big)\Hat{D}^{-1}(\mu).\label{6.13}
\end{equation}

\begin{equation}
\begin{split}
x^2({\zeta}) & \approx C_{2}^{T}\tilde{C}\Hat{\Psi}^{\mu}({\zeta}) =C_{3}^{T}\Hat{\Psi}^{\mu}({\zeta}),
\end{split}\label{6.14}
\end{equation}
where $C_{3}^T=C_{2}\tilde{C}$.\\
Now,
\begin{align*}
\Tilde{p}({\zeta})x^2({\zeta}) & \approx (\Hat{P}^{T}\Hat{\Psi}^{\mu}({\zeta}))(C_{3}^{T}\Hat{\Psi}^{\mu}({\zeta}))^T\\
& = \Hat{P}^{T}\Hat{\Psi}^{\mu}({\zeta})(\Hat{\Psi}^{\mu}({\zeta}))^{T}C_{3},
\end{align*}
\begin{equation}
\Hat{\Psi}^{\mu}({\zeta})(\Hat{\Psi}^{\mu}({\zeta}))^{T}C_{3} \approx \tilde{C}_{4}\Hat{\Psi}^{\mu}({\zeta}),\label{6.15}
\end{equation}
where $\tilde{C}_{4}$ is product operational matrix given as
\begin{equation}
\tilde{C}_{4}=\Big(\int_{0}^{1} (\Hat{\Psi}^{\mu}({\zeta})(\Hat{\Psi}^{\mu}({\zeta}))^{T}C_{3}(\Hat{\Psi}^{\mu}({\zeta}))^{T})d{\zeta}\Big)\Hat{D}^{-1}(\mu).\label{6.16}
\end{equation}
Thus,
\begin{align}\nonumber
~~~~~~~~ \bar{p}({\zeta})x^2({\zeta}) &  \approx \Hat{P}^{T}\tilde{C}_{4}\Hat{\Psi}^{\mu}({\zeta})\\
&=C_{5}^{T}\Hat{\Psi}^{\mu}({\zeta}),\label{6.17} 
\end{align}
where $C_{5}^{T}= \Hat{P}^{T}\tilde{C}_{4}.$\\
Similarly, approximating $u^2({\zeta})$, yields
\begin{align}\nonumber
u^2({\zeta}) &\approx (\Hat{U}^{T} \Hat{\Psi}^{\mu}({\zeta}))^2 \\\nonumber
&  = (\Hat{U}^{T} \Hat{\Psi}^{\mu}({\zeta})) (\Hat{U}^{T} \Hat{\Psi}^{\mu}({\zeta}))^T\\\nonumber
&= \Hat{U}^{T}\Hat{\Psi}^{\mu}({\zeta}) (\Hat{\Psi}^{\mu}({\zeta}))^{T} \Hat{U}\\  \nonumber
& \approx\Hat{ U}^{T} U_{2}\Hat{\Psi}^{\mu}({\zeta})\\
&=U_{3}\Hat{\Psi}^{\mu}({\zeta}),
\label{6.18}  
\end{align}
where $U_{2}$ is product operational matrix given as
\begin{equation*}
    U_{2}=\Big(\int_{0}^{1} (\Hat{\Psi}^{\mu}({\zeta})(\Hat{\Psi}^{\mu}({\zeta}))^{T}\Hat{U}(\Hat{\Psi}^{\mu}({\zeta}))^{T})d{\zeta}\Big)\Hat{D}^{-1}(\mu) ,
\end{equation*}
and $$ U_{3}=\Hat{U}^{T}U_{2}.$$
Now, 
\begin{align}\nonumber
\bar{q}({\zeta})u^2({\zeta}) &\approx (\Hat{Q}^{T} \Hat{\Psi}^{\mu}({\zeta}))(U_{3}^{T}(\Hat{\Psi}^{\mu}({\zeta}))^T \\ \nonumber
&= \Hat{Q}^{T}\Hat{\Psi}^{\mu}({\zeta}) (\Hat{\Psi}^{\mu}({\zeta}))^{T} U_{3}\\ \nonumber
& \approx \Hat{Q}^{T} U_{4}\Hat{\Psi}^{\mu}({\zeta})\\
&=U_{5}^{T}\Hat{\Psi}^{\mu}({\zeta}),
 \label{6.19}
\end{align}
where $U_{4}$ is product operational matrix given as
\begin{equation*}
   U_{4}=\Big(\int_{0}^{1} (\Hat{\Psi}^{\mu}({\zeta})(\Hat{\Psi}^{\mu}({\zeta}))^{T}U_{3}(\Hat{\Psi}^{\mu}({\zeta}))^{T})d{\zeta}\Big)\Hat{D}^{-1}(\mu),
\end{equation*}
and $$ U_{5}^{T}= \Hat{Q}^{T}U_{4}.$$
\\
The cost function $\Tilde{{J}}$ can be approximated as

\begin{align}\nonumber
\Tilde{{J}}\approx\Tilde{{J}}[\Hat{C}, \Hat{U}]=&\dfrac{1}{2}\int_0^1 \Big((\Hat{P}^{T}\Hat{\Psi}^{\mu}({\zeta}))((\Hat{C}^{T} \mathcal{ P}^{\mu} + d_{1}^{T})\Hat{\Psi}^{\mu}({\zeta})(\Hat{\Psi}^{\mu}({\zeta}))^{T}(\Hat{C}^{T} \mathcal{P}^{\mu}  d_{1}^{T})^{T} )\\\nonumber
&+(\Hat{Q}^{T}\Hat{\Psi}^{\mu}({\zeta}))(\Hat{U}^{T}\Hat{\Psi}^{\mu}({\zeta})(\Hat{\Psi}^{\mu}({\zeta}))^{T}\Hat{U})\Big)d{{\zeta}}\\
=&\dfrac{1}{2}\int_{0}^{1}\left(C_{5}^T+U_{5}^T\right)\Hat{\Psi}^{\mu}({\zeta})d{{\zeta}}.\label{6.20}
\end{align}
\\
The dynamical system can be approximated as 
$$ \Hat{C}^{T}\Hat{\Psi}^{\mu}({\zeta})= (\Hat{A}^{T}\Hat{\Psi}^{\mu}({\zeta}))( (\Hat{C}^{T}\mathcal{P}^{\mu} + d_{1}^{T} )\Hat{\Psi}^{\mu}({\zeta})) + (\Hat{B}^{T}\Hat{\Psi}^{\mu}({\zeta}))(\Hat{U}^{T}\Hat{\Psi}^{\mu}({\zeta})).$$
This implies $$ \Hat{C}^{T}\Hat{\Psi}^{\mu}({\zeta})= \Hat{A}^{T} \Tilde{C}\Hat{\Psi}^{\mu}({\zeta})+\Hat{B}^{T} U_{2}\Hat{\Psi}^{\mu}({\zeta}),$$
whence $$(\Hat{C}^{T}-\Hat{A}^{T} \Tilde{C}-\Hat{B}^{T} U_{2})\Hat{\Psi}^{\mu}({\zeta}) =0 .$$
Thus the dynamical system changes into system of algebraic equations:
\begin{equation}
    \Hat{C}^{T}-\Hat{A}^{T} \Tilde{C}-\Hat{B}^{T} U_{2}=0. \label{6.21} 
\end{equation}
Let
\begin{equation}
  \Tilde{{J}}^\ast[\Hat{C}, \Hat{U}, \eta^{\star}]= \Tilde{{J}}[\Hat{C}, \Hat{U}]+ (\Hat{C}^{T}-\Hat{A}^{T} \Tilde{C}-\Hat{B}^{T} U_{2})\eta^{\star}, \label{6.22} 
\end{equation}
where 
$\eta^{\star}=[ \eta_{1}^{\star}, \eta_{2}^{\star}, \dots,\eta_{\Hat{{m}}}^{\star}]^{T}$
is the unknown Lagrange multiplier.\\\\
Now the necessary conditions for the extremum are
\begin{equation}
\dfrac{\partial{\Tilde{{J}}^\star}} {\partial{\Hat{C}}}=0,\label{6.23} 
\end{equation}
\begin{equation}
\dfrac{\partial{\Tilde{{J}}^\star}} {\partial{\Hat{U}}}=0,\label{6.24} 
\end{equation}
\begin{equation}
  \dfrac{\partial{\Tilde{{J}}^\star}} {\partial{\eta^{\star}}}=0. \label{6.25}  
\end{equation}  
\\
The above system of linear  Eqs. \eqref{6.23}, \eqref{6.24}, and \eqref{6.25} have been solved for $\Hat{C}$, $\Hat{U}$, and $\eta^{\star}$ using the Lagrange multiplier method. By determining  $\Hat{C}$ and $\Hat{U}$, we have determined the approximated values of $u({\zeta})$ and $x({\zeta})$ from \eqref{6.5} and \eqref{6.10}, respectively. 
\\
\section{Error estimation }
In this section, the best approximation to a smooth function has been obtained by estimating the error norm using fractional Taylor wavelets.

\begin{lemma}\label{best approx}
Let $f(t)\in C^N[a, b]$ and ${P}_{N-1}({t})$ is the best interpolation polynomial to $f(t)$ at the roots of the $N$-degree shifted Chebyshev polynomial in $[a, b]$. Then 
$$|f(t)-{P}_{N-1}(t)|\leq \dfrac{({b}-{a})^N}{N!2^{2N-1}}\max_{\xi\in[{a}, {b}]} |f^N(\xi)|.$$ 
\end{lemma}
\begin{lemma}\label{wavelets best approx}
Suppose $\displaystyle \sum_{{n}=1}^{2^{{k}-1}}\sum_{{m}=0}^{M-1}c_{{n}{m}}\psi_{{n}{m}}^{\mu}({\zeta})=\displaystyle \sum_{i=1}^{\Hat{m}}\Hat{c}_{i}\Hat{\psi}_{i}^{\mu}({\zeta})=\Hat{C}^{T}\Hat{\Psi}^{\mu}({\zeta})$  be the fractional Taylor wavelets expansion of the real sufficiently smooth function ${f}({\zeta})\in [0,1]$. Then there exists a real number $\Tilde{\mathcal{M}}$ such that 
$$||{f}({\zeta})-\Hat{C}^{T} \Hat{\Psi}^{\mu}({\zeta})||_{2}\leq \dfrac{\Tilde{\mathcal{M}}}{\Hat{m}!2^{2\Hat{m}-1}}.$$
\end{lemma}
\begin{proof}
We are able to write,
$$\int_{0}^{1}\left({f}({\zeta})-\Hat{C}^{T} \Hat{\Psi}^{\mu}({\zeta})\right)^2d{\zeta}= \sum_{{n}=1}^{2^{{k}-1}M}\int_{\left(\frac{{n}-1}{2^{{k}-1}}\right)^{\frac{1}{\mu}}}^{\left(\frac{{n}}{2^{{k}-1}}\right)^{\frac{1}{\mu}}}\left({f}({\zeta})-\Hat{C}^{T} \Hat{\Psi}^{\mu}({\zeta})\right)^2d{\zeta},$$
On the subinterval $\left[\left(\frac{{n}-1}{2^{{k}-1}}\right)^{\frac{1}{\mu}},\left(\frac{{n}}{2^{{k}-1}}\right)^{\frac{1}{\mu}}\right)$, $\Hat{C}^{T}\Hat{\Psi}^{\mu}({\zeta})$
is a polynomial of degree at most $\Hat{m}= 2^{{k}-1}M$, that approximates ${f}$ with the least-square property.\\\\
As ${P}_{\Hat{m}-1}({\zeta})$ is the best approximation to ${f}({\zeta})$, that agrees with ${f}({\zeta})$ at the zeros of shifted Chebyshev polynomial. Therefore, the following has been obtained by using Lemma \ref{best approx}:
\begin{equation*}
\begin{split}
||{f}({\zeta})-\Hat{C}^{T} \Hat{\Psi}^{\mu}({\zeta})||_{2}&=\sum_{{n}=1}^{2^{{k}-1}M}\int_{\left(\frac{{n}-1}{2^{{k}-1}}\right)^{\frac{1}{\mu}}}^{\left(\frac{{n}}{2^{{k}-1}}\right)^{\frac{1}{\mu}}}\left({f}({\zeta})-\Hat{C}^{T} \Hat{\Psi}^{\mu}({\zeta})\right)^2d{\zeta} \\
& \leq\sum_{{n}=1}^{2^{{k}-1}M}\int_{\left(\frac{{n}-1}{2^{{k}-1}}\right)^{\frac{1}{\mu}}}^{\left(\frac{{n}}{2^{{k}-1}}\right)^{\frac{1}{\mu}}}\left({f}({\zeta})-{P}_{\Hat{m}-1}({\zeta})\right)^2d{\zeta} \\
&\leq \int_{0}^{1}\Big(\dfrac{1}{\Hat{m}!2^{2\Hat{m}-1}}\max_{\xi\in [0, 1]}|{f}^{(\Hat{m})}(\xi)|\Big)^2 d{\zeta}.
\end{split}   
\end{equation*}
Let us consider that there exists a real number $\Tilde{\mathcal{M}}$ such that
$$\max_{\xi\in[0, 1]} |{f}^{(\Hat{m})}(\xi)| \leq \Tilde{\mathcal{M}}.$$
Thus,
\begin{align*}
\begin{split}
||{f}({\zeta})-\Hat{C}^{T}\Hat{\Psi}^{\mu}({\zeta})||_{2}^2 & = \int_{0}^{1}\left({f}({\zeta})-\Hat{C}^{T} \Hat{\Psi}^{\mu}({\zeta})\right)^2d{\zeta}  \\
&\leq \int_{0}^{1}\Big(\dfrac{\Tilde{\mathcal{M}}}{\Hat{m}!2^{2\Hat{m}-1}}\Big)^2 d{\zeta}.\\
\end{split}    
\end{align*}
This implies that
\begin{align*}
||{f}({\zeta})-\Hat{C}^{T} \Hat{\Psi}^{\mu}({\zeta})||_{2} \leq \dfrac{\Tilde{\mathcal{M}}}{\Hat{m}!2^{2\Hat{m}-1}}.
\end{align*}
\end{proof}
Now an error estimate concerning   $\displaystyle |\inf_{\Gamma_{\Hat{m}}} \Tilde{J}-\inf_{\Gamma}\Tilde{J}|$  have been established   for the proposed method.
By considering $y({\zeta})= {_{0}^C}\mathcal{D}_{{\zeta}}^{\mu}x({\zeta})$, the problem Eqs. \eqref{1.1} and \eqref{1.2} is equivalent to the following problem:
\begin{equation}
\begin{split}
\min\Tilde{J} & = \int_{0}^{1} \mathfrak{F}({\zeta}, \mathcal{I}^{\mu}y({\zeta})+x_{0}, u({\zeta})) d{\zeta}\\
&= \int_{0}^{1} \mathfrak{F}({\zeta}, y({\zeta}), u({\zeta})) d{\zeta},
\end{split}\label{7.1}
\end{equation}
\begin{equation}
  y({\zeta})=\Tilde{{a}}({\zeta})(\mathcal{I}^{\mu}y({\zeta})+x_{0})+\Tilde{{b}}({\zeta})u({\zeta}).\label{7.2}
\end{equation}
\\
\begin{theorem}
 The set $\Gamma$ is consisting of all functions $(y({\zeta}),u({\zeta}))$ that satisfy Eq. \eqref{7.2} and $\Gamma_{\Hat{m}}$ is a subset of $\Gamma$ consisting of all functions   $\displaystyle \left(\sum_{j=1}^{\Hat{m}}x_{j}^{\star}\psi_{j}^{\mu}({\zeta}), \sum_{j=1}^{\Hat{m}}u_{j}^{\star}\psi_{j}^{\mu}({\zeta})\right)$. Then there exist real
numbers $\Tilde{\mathcal{M}}^{\star}$ and $\Tilde{\mathcal{L}} >0$, such that
\begin{align*}
\displaystyle|\inf_{\Gamma_{\Hat{m}}}\Tilde{J} - \displaystyle\inf_{\Gamma}\Tilde{J}|\leq \dfrac{\Tilde{\mathcal{L}}\Tilde{\mathcal{M}}^{\star}}{\Hat{m}!2^{2\Hat{m}-1}}.
\end{align*}
\end{theorem}
\begin{proof}
Assume that $\mathfrak{F}$ is Lipschitz  continuous function with respect to $(y, u)$ with a Lipschitz constant $\Tilde{\mathcal{L}}$.
\begin{align}
 |\mathfrak{F}({\zeta}, y, u)-\mathfrak{F}({\zeta}, \Tilde{y}, \Tilde{u})|\leq \Tilde{\mathcal{L}}(||y-\Tilde{y}||_{2}+||u-\Tilde{u}||_{2}),\label{7.3}
 \end{align}   
for all ${\zeta}\in [0,1]$, for all $\Tilde{y}, \Tilde{u}, y, u $  belong to $L^2[0,1]$. Since $\Gamma_{\Hat{m}} \subseteq \Gamma $,
 therefore $\displaystyle\inf_{\Gamma_{\Hat{m}}}\Tilde{J}\geq $ 
 $\displaystyle\inf_{\Gamma} \Tilde{J} $.
 \\\\
Let $(y, u)$ be the exact solution for state and control functions
and $\displaystyle\left(\sum_{j=1}^{\Hat{m}}x_{j}^{\star}\psi_{j}^{\mu}({\zeta}), \sum_{j=1}^{\Hat{m}}u_{j}^{\star}\psi_{j}^{\mu}({\zeta})\right)$ be the fractional Taylor wavelets  expansion of $(y, u)$, then
\begin{equation*}
\begin{split}
 |\displaystyle\inf_{\Gamma_{\Hat{m}}}\Tilde{J} - \displaystyle\inf_{\Gamma}\Tilde{J}| \leq& \left|\Tilde{J}\left(\sum_{j=1}^{\Hat{m}}x_{j}^{\star}\psi_{j}^{\mu}({\zeta}), \sum_{j=1}^{\Hat{m}}u_{j}^{\star}\psi_{j}^{\mu}({\zeta})\right)-\Tilde{J}(y,u)\right|\\  
=&\left|\int_{0}^{1}\mathfrak{F}\left({\zeta}, \sum_{j=1}^{\Hat{m}}x_{j}^{\star}\psi_{j}^{\mu}({\zeta}), \sum_{j=1}^{\Hat{m}}u_{j}^{\star}\psi_{j}^{\mu}({\zeta})\right)d{\zeta} -\int_{0}^{1} \mathfrak{F}({\zeta}, y({\zeta}), u({\zeta}))d{\zeta}\right| \\
\leq &\int_{0}^{1}\left|\mathfrak{F}\left({\zeta}, \sum_{j=1}^{\Hat{m}}x_{j}^{\star}\psi_{j}^{\mu}({\zeta}), \sum_{j=1}^{\Hat{m}}u_{j}^{\star}\psi_{j}^{\mu}({\zeta})\right)- \mathfrak{F}\left({\zeta},\sum_{j=1}^{\Hat{m}}x_{j}^{\star}\psi_{j}^{\mu}({\zeta}),u({\zeta})\right)\right|d{\zeta}\\
&+  \int_{0}^{1} \left|\mathfrak{F}\left({\zeta},\sum_{j=1}^{\Hat{m}}x_{j}^{\star}\psi_{j}^{\mu}({\zeta}), u({\zeta}))\right) - \mathfrak{F}({\zeta}, y({\zeta}), u({\zeta})) \right|d{\zeta}.\\
\end{split}
\end{equation*}
By using the Lemmas \ref{best approx} and \ref{wavelets best approx} , it follows:
\begin{equation*}
\begin{split}
|\displaystyle\inf_{\Gamma_{\Hat{m}}}\Tilde{J} - \displaystyle\inf_{\Gamma}\Tilde{J}| & \leq \int_{0}^{1} \Tilde{\mathcal{L}} \left(||y-  \sum_{j=1}^{\Hat{m}}x_{j}^{\star}\psi_{j}^{\mu}({\zeta})||_{2}+||u-\sum_{j=1}^{\Hat{m}}u_{j}^{\star}\psi_{j}^{\mu}(t)||_{2}\right)d{\zeta}\\
 & \leq \Tilde{\mathcal{L}} \left(\dfrac{\Tilde{\mathcal{M}}_{1}}{\Hat{m}!2^{2\Hat{m}-1}}+\dfrac{\Tilde{\mathcal{M}}_{2}}{\Hat{m}!2^{2\Hat{m}-1}}\right),\\
\end{split}  
\end{equation*}
where $\max_{{\zeta}\in[0,1]}|y^{(\Hat{m})}({\zeta})|=\Tilde{\mathcal{M}}_{1}$ 
and  $\max_{{\zeta}\in[0,1]}|u^{(\Hat{m})}({\zeta})|= \Tilde{\mathcal{M}}_{2}$.\\
Hence
\begin{align*}
|\displaystyle\inf_{\Gamma_{\Hat{m}}}\Tilde{J} - \displaystyle\inf_{\Gamma}\Tilde{J}|\leq \dfrac{\Tilde{\mathcal{L}}\Tilde{\mathcal{M}}^{\star}}{\Hat{m}!2^{2\Hat{m}-1}}, 
\end{align*}
where  $\Tilde{\mathcal{M}}^{\star}= \Tilde{\mathcal{M}}_{1}+ \Tilde{\mathcal{M}}_{2}$.\\\\
\end{proof}

\section{ Convergence analysis}
In this section, convergence analysis has been analyzed.
\begin{theorem}
The approximate solutions $y({\zeta})\approx\Hat{Y}^T\Hat{\Psi}^{\mu}({\zeta})$ and $u({\zeta})\approx \Hat{U}^T\Hat{\Psi}^{\mu}({\zeta})$, converge respectively to the exact solution as $\Hat{m}$, the number of the fractional Taylor wavelets tends to infinity.\\
\end{theorem}
 \begin{proof} 
 Suppose $\Gamma _{\Hat{m}}$ is the set of all $\left(\Hat{Y}^T \Hat{\Psi}^{\mu}({\zeta}), \Hat{U}^T\Hat{\Psi}^{\mu}({\zeta})\right) $ which satisfies the constraint Eq. \eqref{7.2}.\\\\
Using convergence property of fractional Taylor wavelets, for each $\left(\Hat{Y}_{1}\Hat{\Psi}^{\mu}({\zeta}),\Hat{U}_{1}\Hat{\Psi}^{\mu}({\zeta})\right)\in \Gamma_{\Hat{m}} $, there exists a unique pair of functions 
 $(y_{1}({\zeta}),  u_{1}({\zeta})) $  such that
 $$\left(\Hat{Y}_{1}^T\Hat{\Psi}^{\mu}({\zeta}), \Hat{U}_{1}^T\Hat{\Psi}^{\mu}({\zeta})\right)\to (y_{1}({\zeta}),  u_{1}({\zeta})),$$ 
 as $ \Hat{m}\to \infty $.
 It is obvious that $(y_{1}({\zeta}),  u_{1}({\zeta}))\in \Gamma,$ where $\Gamma$ is the set of all solutions that satisfy the constraint given in Eq. \eqref{7.2}. So as $\Hat{m}\to \infty,$ each element in $\Gamma_{\Hat{m}}$ tends to an element in $\Gamma$.\\\\
 Furthermore,  as $\Hat{m}\to \infty$, then \\
   $$\Tilde{J}_{1}^{\Hat{m}} = \Tilde{J}\left(\Hat{Y}_{1}^T\Hat{\Psi}^{\mu}({\zeta}), \Hat{U}_{1}^T\Hat{\Psi}^{\mu}({\zeta})\right)\to \Tilde{J}_{1},$$
 where $\Tilde{J}_{1}^{\Hat{m}}$ is the value of the cost
function corresponding to the pair $\left(\Hat{Y}_{1}^T\Hat{\Psi}^{\mu}({\zeta}), \Hat{U}_{1}^T\Hat{\Psi}^{\mu}({\zeta})\right)$ and $\Tilde{J}_{1}$ is the objective value corresponding to the feasible solution  $({y}_{1}({\zeta}),  {u}_{1}({\zeta}))$. \\\\
Now 
\begin{align*}
\Gamma_{1} \subseteq \Gamma_{2} \subseteq \dots \subseteq \Gamma_{\Hat{m}}\subseteq \Gamma_{\Hat{m}+1}\subseteq \dots\subseteq \Gamma ,
\end{align*} 
then
\begin{align*}
\inf_{\Gamma_{1}}{\Tilde{J}_{1}} \geq \inf_{\Gamma_{2}}{\Tilde{J}_{2}} \geq \dots \geq \inf_{\Gamma_{\Hat{m}}}{\Tilde{J}_{\Hat{m}}}\geq \inf_{\Gamma_{\Hat{m}+1}}{\Tilde{J}_{\Hat{m}+1}} \geq \dots\geq\inf_{\Gamma}{\Tilde{J}},
\end{align*}
which is a non-increasing and bounded sequence. Since every bounded monotone sequence is convergent.
Therefore, it converges to a number $\omega$, such that
$$ \omega \geq \inf_{\Gamma}\Tilde{J}.$$\\
Next, we have to show that 
$$ \omega=\lim_{\Hat{m}\to \infty}\inf_{\Gamma_{\Hat{m}}}{\Tilde{J}_{\Hat{m}}} =\inf_{\Gamma}\Tilde{J}. $$
Given $\epsilon>0$, let $(y(\zeta), u(\zeta)) \in \Gamma$\\
such that 
\begin{equation}
  \Tilde{J}(y({\zeta}), {u}({\zeta})) < \inf_{\Gamma}\Tilde{J}+\epsilon .\label{8.1}
\end{equation}
Since $\Tilde{J}(y({\zeta}),  u({\zeta}))$ is continuous, for this value
of $\epsilon$, there exists $K(\epsilon)$ such that if $\Hat{m}>K(\epsilon)$\\
\begin{align*}
\left|\Tilde{J}(y({\zeta}), u({\zeta}))-\Tilde{J}\left(\Hat{Y}\Hat{\Psi}^{\mu}({\zeta}), \Hat{U}\Hat{\Psi}^{\mu}({\zeta})\right)\right|< \epsilon .
\end{align*} 
This implies that
\begin{equation}
\Tilde{J}\left(\Hat{Y}\Hat{\Psi}^{\mu}({\zeta}), \Hat{U}\Hat{\Psi}^{\mu}({\zeta})\right) < \Tilde{J}(y({\zeta}), u({\zeta}))+\epsilon.
\label{8.2} 
\end{equation}
Using Eqs. \eqref{8.1} and \eqref{8.2}, it follows that
\begin{align}
\Tilde{J}\left(\Hat{Y}\Hat{\Psi}^{\mu}({\zeta}), \Hat{U}\Hat{\Psi}^{\mu}({\zeta})\right) < \Tilde{J}(y({\zeta}), u({\zeta}))+\epsilon  < \inf_{\Gamma}\Tilde{J}+2\epsilon.  \label{8.3} 
\end{align}
\\
Now
\begin{align}
\inf_{\Gamma}\Tilde{J}\leq  \inf_{\Gamma_{\Hat{m}}}\Tilde{J}_{\Hat{m}}\leq \Tilde{J}\left(\Hat{Y}\Hat{\Psi}^{\mu}({\zeta}), \Hat{U}\Hat{\Psi}^{\mu}({\zeta})\right).\label{8.4}
\end{align}
From Eqs. \eqref{8.3} and \eqref{8.4}, one get
\begin{align*}
   \inf_{\Gamma}\Tilde{J}\leq  \inf_{\Gamma_{\Hat{m}}}\Tilde{J}_{\Hat{m}}<\inf_{\Gamma}\Tilde{J}+2\epsilon.  
\end{align*}
Whence
$$
  0\leq \inf_{\Gamma_{\Hat{m}}}\Tilde{J}_{\Hat{m}}-\inf_{\Gamma}\Tilde{J}< 2\epsilon,$$ 
where $\epsilon$ is chosen arbitrary.\\
Hence
$$ \omega=\lim_{\Hat{m}\to \infty}\inf_{\Gamma_{\Hat{m}}}{\Tilde{J}_{\Hat{m}}} =\inf_{\Gamma}\Tilde{J}.$$

\end{proof}
\section{Numerical Examples}
In this section, the accuracy and utility of the proposed method are demonstrated with some numerical examples.
\begin{example}\label{1st example}
Consider the following time invariant problem:
$$\min\Tilde{{J}}=\dfrac{1}{2}\int_{0}^{1}\left(x^2({\zeta})+u^2({\zeta})\right)d{\zeta},$$
  
    $$_0^{C}\mathcal{D}_{{\zeta}}^{\mu}x({\zeta})=-x({\zeta})+u({\zeta}),$$
   
    $$x(0)=1.$$
\end{example}
Our aim is to find  state function $x({\zeta})$ and control function $u({\zeta})$   which minimize the cost function $\Tilde{{J}}$.  The exact solution of the above example for
$\mu = 1$ as follows:
\begin{equation}
    x({\zeta})=\cosh(\sqrt{2}{\zeta})+\varpi\sinh(\sqrt{2}{\zeta}), \label{9.1}
\end{equation}
\begin{equation}
u({\zeta})=(1+\sqrt{2}\varpi)\cosh(\sqrt{2}{\zeta})+(\sqrt{2}+\varpi)\sinh(\sqrt{2}{\zeta}) ,\label{9.2}
\end{equation}
where $\varpi\approx-0.98$.\\

This FOCP has been solved by the proposed method using fractional Taylor wavelets and Taylor wavelets with ${M} = 4$ and ${{k}} = 2$. Table \ref{table: 1} shows absolute error of state and control functions when $\mu =1$ and different values of $\zeta $. The approximate solutions obtained by TW and FTW methods for different values of $\mu$ for state function $x(\zeta)$ and control function $u(\zeta)$ are shown in Tables \ref{table: 2} and \ref{table: 3}, respectively. Table \ref{table:4} shows the approximate values
of cost function $\Tilde{{J}}$ obtained by  TW and FTW for  $\mu=0.85, 0.95, 0.99 $ and $ 1$. Figs. 1 and  2, respectively, show the exact and approximate solutions of the state function and control function when $\mu=1$ using fractional Taylor wavelet and Taylor wavelet. Figs. 3 and 4 show the approximate solutions of the state function and control function, respectively, for diverse values of $\mu$ by using the  Taylor wavelet method. Similarly, Figs. 5 and  6 show the approximate solutions of the state function and control function, respectively, for diverse values of $\mu$ by using the fractional Taylor wavelet method, and  it has been observed that there is a good agreement with the obtained solutions.\\

\includegraphics[scale=.77]{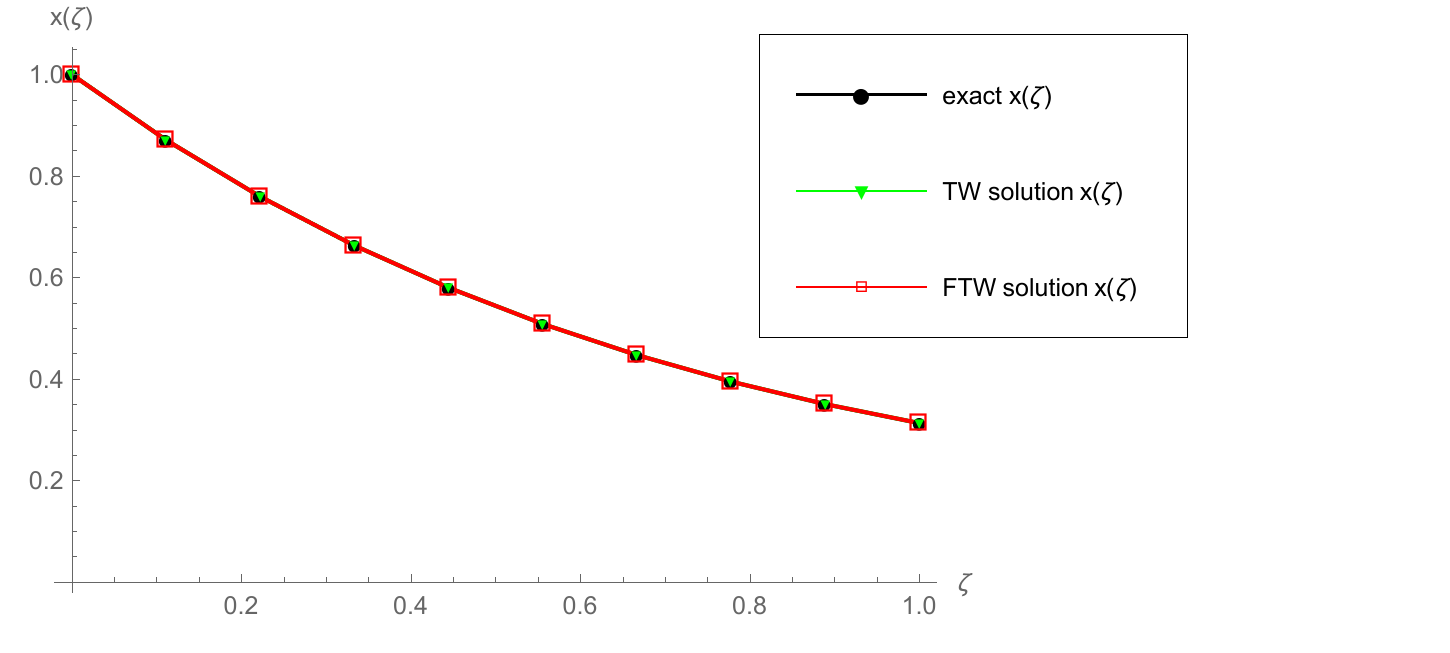}

Fig. 1: Exact , TW and FTW solutions of $x({\zeta})$ when $\mu=1.$
\\\\

\includegraphics[scale=.8]{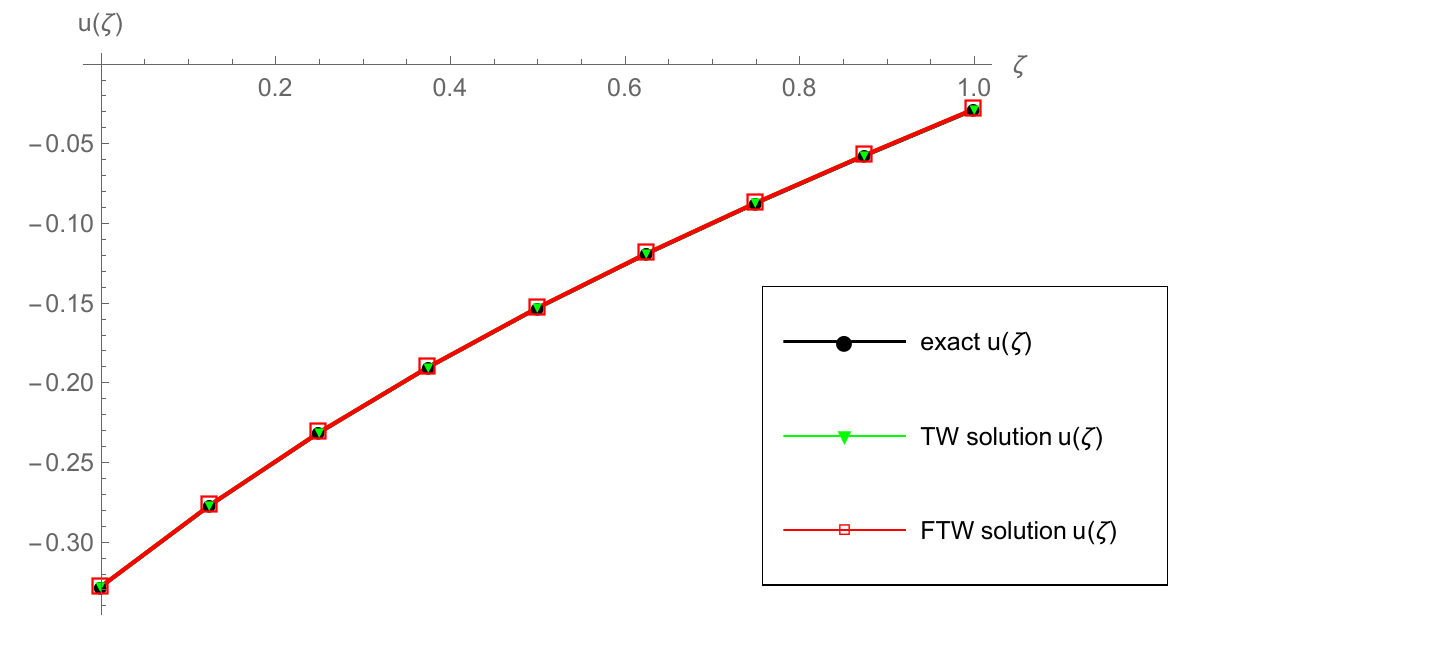}

Fig. 2: Exact, TW and FTW solutions of $u({\zeta})$ when $\mu=1.$\\

\begin{table}
\centering
\caption{Absolute error of state and control functions when $\mu=1.$} 
\begin{tabular}{|c|c|c|} \hline
${\zeta}$ &  Absolute error of  $x({\zeta})$  &  Absolute error of  $u({\zeta})$ \\

\hline
0.1 & $5.35737\times10^{-5}$ &  $1.48744\times10^{-4}$\\
\hline
0.2 & $4.43828\times10^{-6}$ & $1.07836\times10^{-4}$   \\
\hline
0.3 & $1.20259\times10^{-5}$ & $1.4746\times10^{-4}$ \\
\hline
0.4 & $9.07815\times10^{-5}$ & $1.59711\times10^{-4}$   \\
\hline
0.5  &  $4.26262\times10^{-6}$ & $2.10423\times10^{-4}$ \\
\hline
0.6 & $9.69844\times10^{-5}$ & $2.23537\times10^{-4}$  \\
\hline
0.7 & $7.68442\times10^{-5}$ & $2.63508\times10^{-4} $\\
\hline
0.8  & $9.69206\times10^{-5}$ & $2.9988\times10^{-4}$  \\
\hline
0.9 &  $1.51849\times10^{-4}$ & $3.37278\times10^{-4}$\\
\hline
\end{tabular}
\label{table: 1}
\end{table}

\includegraphics[scale=0.75]{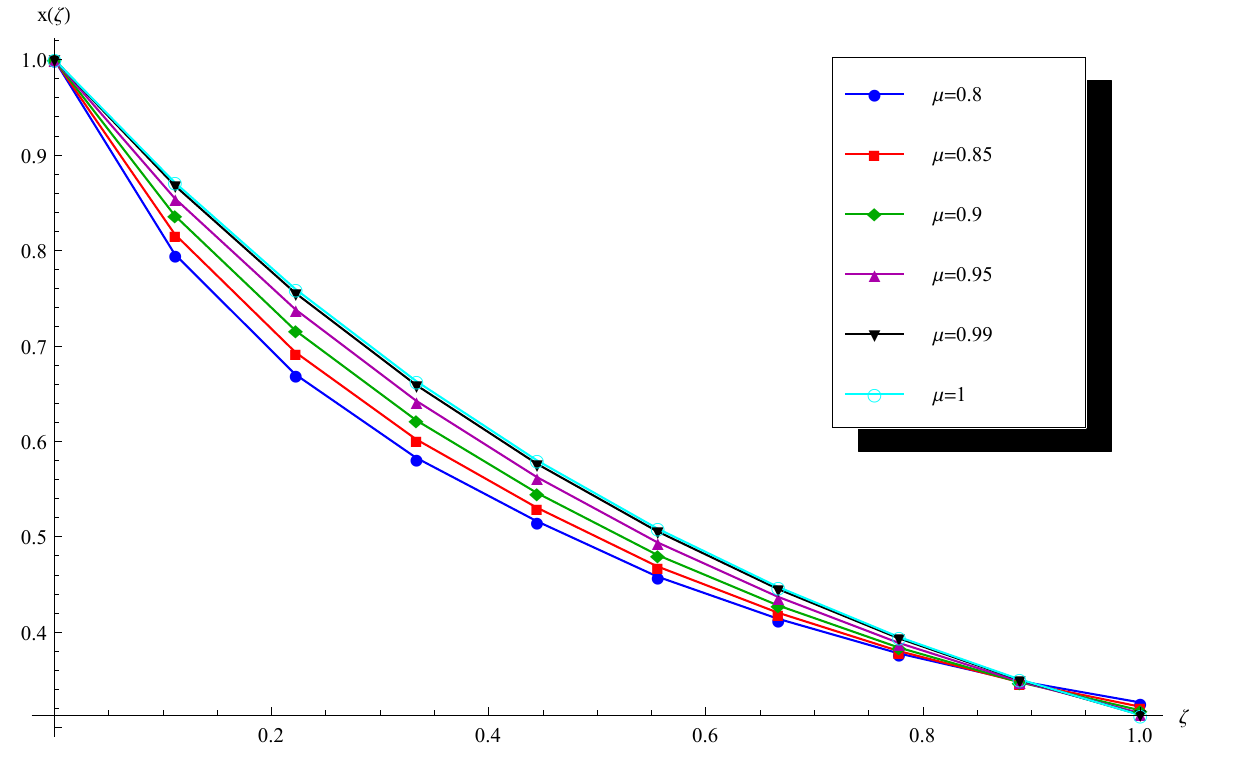}

Fig. 3: Using TW  solutions of $x({\zeta})$ for $\mu = 0.8, 0.85, 0.9, 0.95, 0.99$ and $ 1.$\\

\includegraphics[scale=0.75]{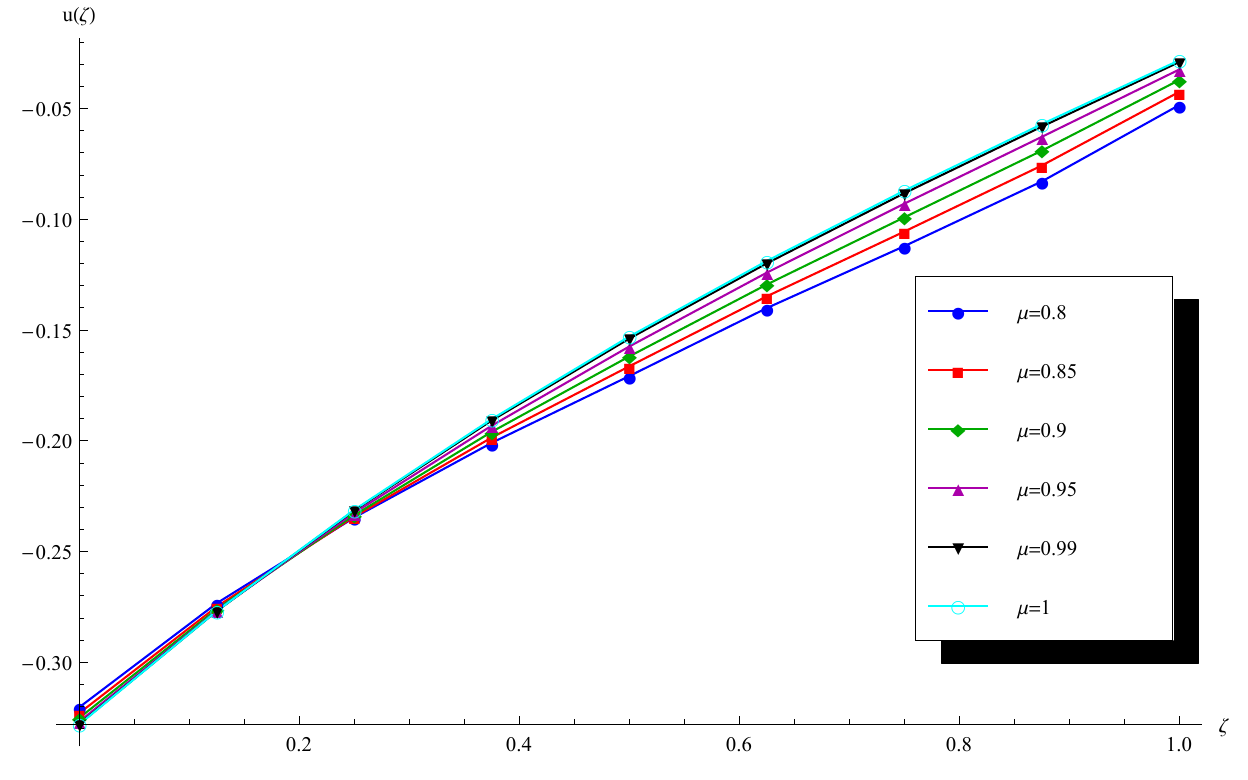}

Fig. 4: Using TW  solutions of $u({\zeta})$ for $\mu = 0.8, 0.85, 0.9, 0.95, 0.99$ and $ 1.$
\\
\includegraphics[scale=0.75] {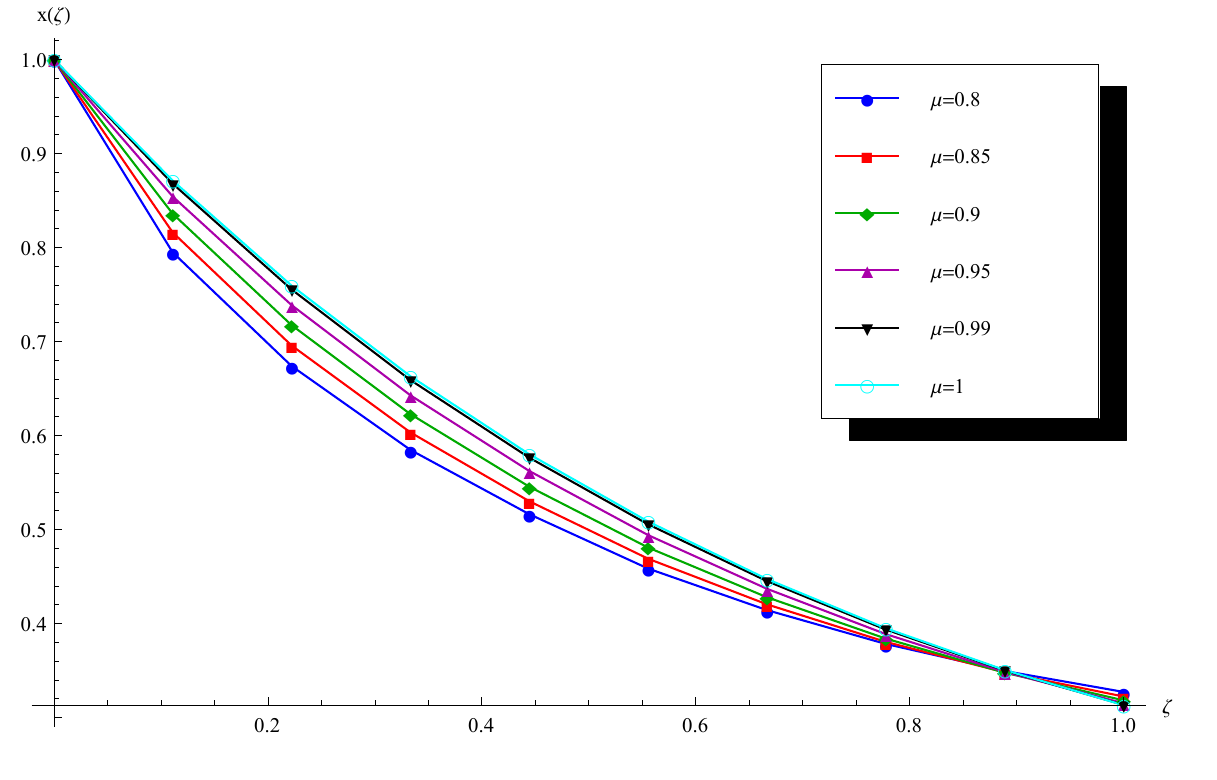}

Fig. 5: Using FTW  solutions of $x({\zeta})$ for $\mu = 0.8, 0.85, 0.9, 0.95, 0.99$ and $ 1.$\\

\begin{landscape}
\begin{table}
\centering
\caption{Approximate results of $x({\zeta})$ for $\mu=0.5, 0.75, 0.85, 0.95$ and $ 0.99.$} 
\begin{tabular}{|c|c|c|c|c|c|c|c|c|c|c|c|c|c|c|c|c|c} \hline
${\zeta}$ &  \multicolumn{2}{|c|}{$\mu=0.5$} &\multicolumn{2}{|c|}{$\mu=0.75$} &\multicolumn{2}{|c|}{$\mu=0.85$}& \multicolumn{2}{|c|}{$\mu=0.95$} &\multicolumn{2}{|c|}{$\mu=0.99$}\\[0.5ex] \cline{1-11}
& TW & FTW& TW &  FTW & TW & FTW & TW & FTW & TW & FTW \\[0.5ex] \cline{2-11}
\hline
0.1 & 0.695033 & 0.706716 &  0.826483 & 0.826205 & 0.866578 & 0.866054 & 0.899198  & 0.898982 & 0.910304 & 0.91026\\
\hline
0.2 &  0.585076 & 0.622558 &  0.729052 & 0.734432 & 0.779992 &  0.782106 & 0.824404 &  0.824876 & 0.840254 & 0.840334\\
\hline
0.3 & 0.534201 & 0.55043 &   0.670491  & 0.675079 & 0.722854  &0.724269  &0.770655 & 0.770923 & 0.78828 & 0.788325\\
\hline
0.4 & 0.508309 & 0.535243 &  0.636575 & 0.633937 &0.687537 &  0.688044 & 0.735357 & 0.735274 & 0.753365 & 0.753341\\
\hline
0.5 & 0.479621 & 0.520989  & 0.616584 & 0.622424 & 0.66812 & 0.670137 & 0.716283 & 0.71666 &  0.734461 & 0.734521\\
\hline
0.6 & 0.497024 & 0.528249 &  0.622844 & 0.62608 & 0.670801  & 0.671888 & 0.716051 &  0.716253 & 0.733272 & 0.733304\\
\hline
0.7 & 0.523824  &0.567539 &  0.647981 & 0.651046 & 0.691985 & 0.692859 & 0.733308 & 0.733455 & 0.749074 & 0.749097\\
\hline
0.8  & 0.585355 & 0.644608  & 0.698397 & 0.702236 & 0.735577 & 0.73656 & 0.770585 & 0.770731 & 0.784045 & 0.784067\\
\hline
0.9 & 0.70695 &  0.762689 &  0.780492  & 0.783741 & 0.805481  &0.806219  & 0.830413 & 0.830503 &  0.840366 & 0.840378\\
\hline
\end{tabular}

\footnotesize{TW: Taylor wavelet method; FTW: Fractional Taylor Wavelet method.}\\
\label{table: 2}
\end{table}
\begin{table}
\centering
\small\addtolength{\tabcolsep}{-2pt}
\caption{Approximate results of $u({\zeta})$ for $\mu=0.2,0.5, 0.75, 0.85, 0.95$ and $ 0.99.$ }
\begin{tabular}{|c|c|c|c|c|c|c|c|c|c|c|c|c|c|c|c|} \hline
${\zeta}$ &  \multicolumn{2}{|c|}{$\mu=0.5$} &\multicolumn{2}{|c|}{$\mu=0.75$} &\multicolumn{2}{|c|}{$\mu=0.85$}& \multicolumn{2}{|c|}{$\mu=0.95$} &\multicolumn{2}{|c|}{$\mu=0.99$}\\[0.5ex] \cline{1-11}
& TW & FTW & TW & FTW & TW & FTW & TW & FTW & TW & FTW \\ \cline{2-11}
\hline
0.1   & -0.293625 & -0.29615 &  -0.317214 & -0.317391 & -0.322982 & -0.322988 & -0.326832 &  -0.326767 & -0.327862 & -0.327842 \\ 
\hline
0.2 & -0.25459 & -0.257787 &  -0.271352 & -0.272044 &-0.27493 & -0.275297 & -0.276664  &-0.276797 & -0.276859 & -0.276886\\ 
\hline
0.3 &  -0.230403 & -0.235918 &   -0.23474 & -0.234353 & -0.234314 & -0.234108 & -0.232576 & -0.232552 & -0.231521 & -0.231522\\ 
\hline
0.4 &  -0.210476 & -0.208004 &  -0.203129 & -0.20592 & -0.19858 & -0.197934 & -0.193204 & -0.192989 & -0.190848 & -0.190806\\ 
\hline
0.5  & -0.194217 & -0.191504 &   -0.175169 & -0.172881 &-0.16643 & -0.165739 & -0.157491 & -0.157481 & -0.153912 & -0.153924\\ 
\hline
0.6   &-0.168319 & -0.17554  & -0.14534 & -0.145906 & -0.134854 & -0.134901 & -0.124171 &  -0.124165 & -0.11995 & -0.119949 \\ 
\hline
0.7 & -0.151406 &  -0.154272 &   -0.118895 & -0.119564 & -0.10562 & -0.105762 &  -0.0930333 & -0.0930285 & -0.0883 & -0.0882968 \\ 
\hline
0.8   & -0.131593 & -0.124293 &  -0.0905718 & -0.0896721 & -0.0758178 & -0.0755448 & -0.0628664 & -0.0628195 & -0.0582497 & -0.0582424\\ 
\hline
0.9   & -0.0969982 & -0.0835134 &   -0.0551091 & -0.0528598 & -0.0425355 & -0.0418162 & -0.0324607 & -0.0323489 & -0.029085 &  -0.0290693\\ 
\hline
\end{tabular}
\label{table: 3}
\end{table}

\end{landscape}

\begin{table}[h]
\centering
\caption{Approximate values of cost function $\Tilde{{J}}$ for different values of  $\mu$}
\begin{tabular}{|c|c|c|}\hline
$\mu $& TW     & FTW \\ \cline{1-3}
\hline
1 & 0.192909 & 0.192909\\
\hline
0.99 & 0.191531 & 0.191541\\
\hline
0.95 & 0.186105 &  0.186167 \\
\hline
0.85 & 0.173184 &0.173487\\
\hline
0.75 & 0.161202 & 0.162042\\
\hline
0.5 & 0.135314 &0.141662\\
\hline
 \end{tabular}
 \label{table:4}
\end{table}

\includegraphics[scale=0.75]{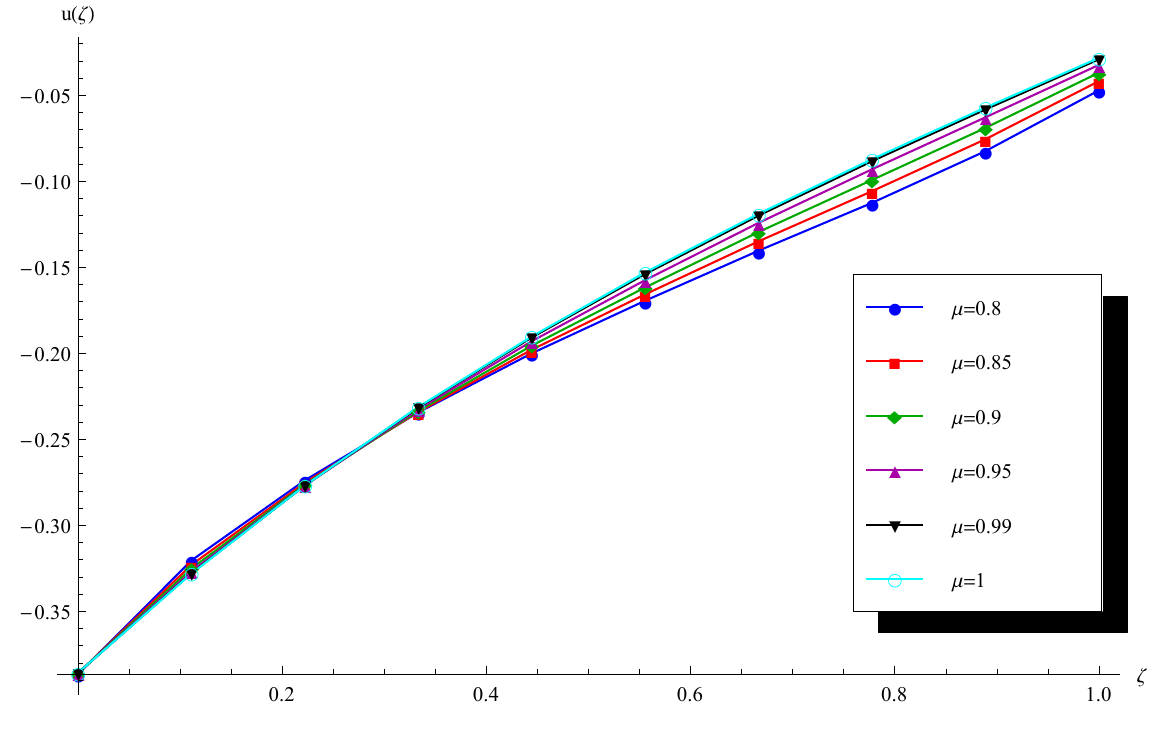}

Fig. 6: Using FTW  solutions of $u({\zeta})$ for $\mu = 0.8, 0.85, 0.9, 0.95, 0.99$ and $ 1.$
\\

\begin{example}\label{2nd example}
  Consider the following time invariant problem:
    $$\min\Tilde{{J}} = \dfrac{1}{2}\int_{0}^{1}\left(x^2({\zeta})+u^2({\zeta})\right)d{\zeta},$$
 $$_0^{C}\mathcal{D}_{{\zeta}}^{\mu}x({\zeta}) = tx({\zeta})+u({\zeta}),$$
 
    $$x(0)=1.$$     
\end{example}
Example \ref{2nd example} has been solved by the proposed method using a fractional Taylor wavelet and a Taylor wavelet with ${M} = 4$ and $ {{k}} = 2 $ for different values of $\mu$. The approximate solutions of state function $ x({\zeta}) $ and control function  $u({\zeta})$ are shown in Tables \ref{table:5} and \ref{table:6}, respectively. Figs. 7 and 8 show the approximate solutions of the state function and control function obtained by using the fractional Taylor wavelet method for various values of $\mu$. Similarly, Figs. 9 and 10 show the approximate solutions of the state function and control function obtained by using the  Taylor wavelet method for various values of $\mu$.

\includegraphics[scale=0.8]{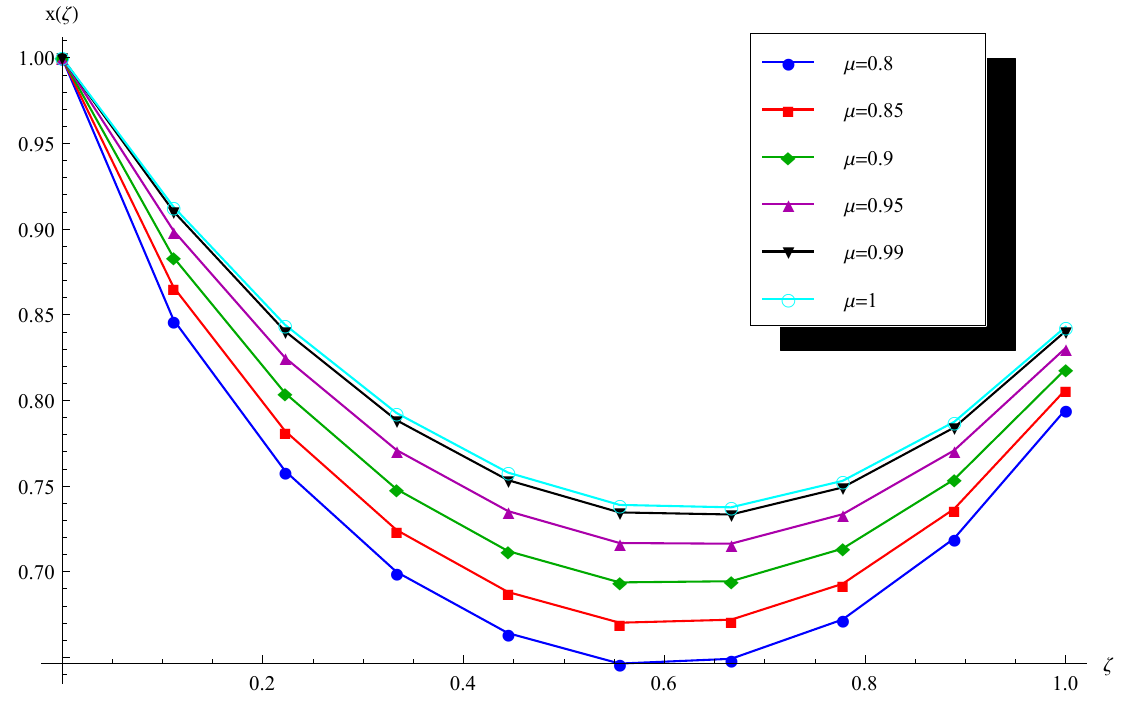}

Fig. 7: Using TW  solutions of $x({\zeta})$ for  $\mu=0.8, 0.85, 0.9, 0.95, 0.99$ and $ 1$.
\\\\

\includegraphics[scale=0.8]{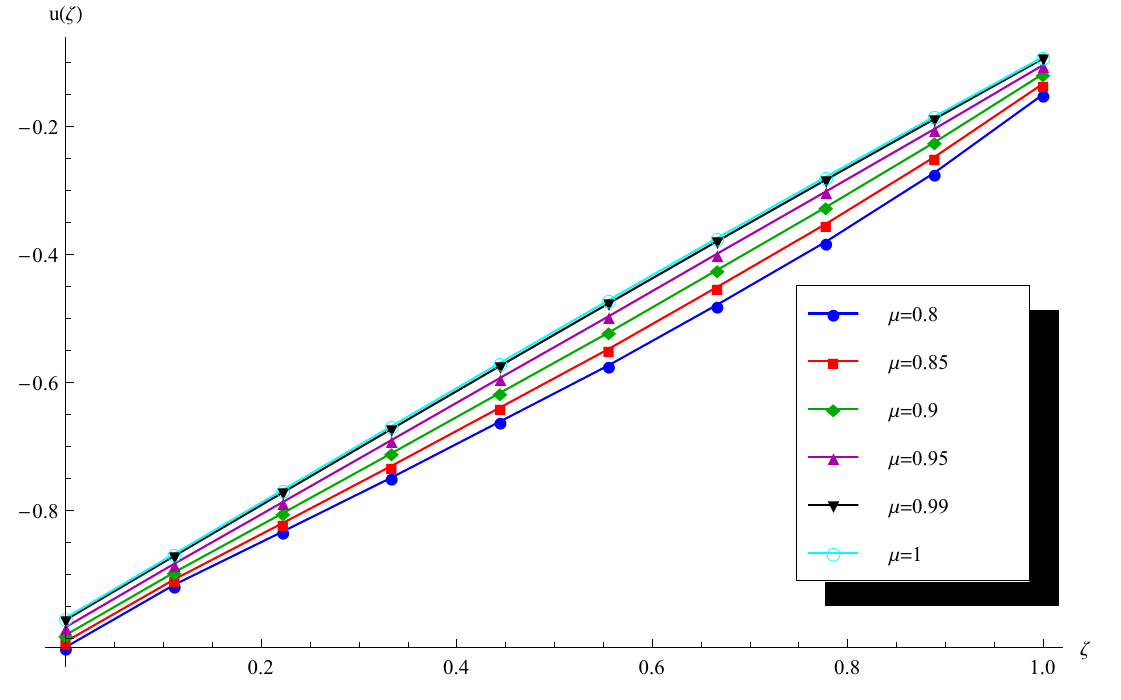}

Fig. 8: Using TW  solutions of $u({\zeta})$ for  $\mu=0.8, 0.85, 0.9, 0.95, 0.99$ and $ 1$.\\

\includegraphics[scale=0.8]{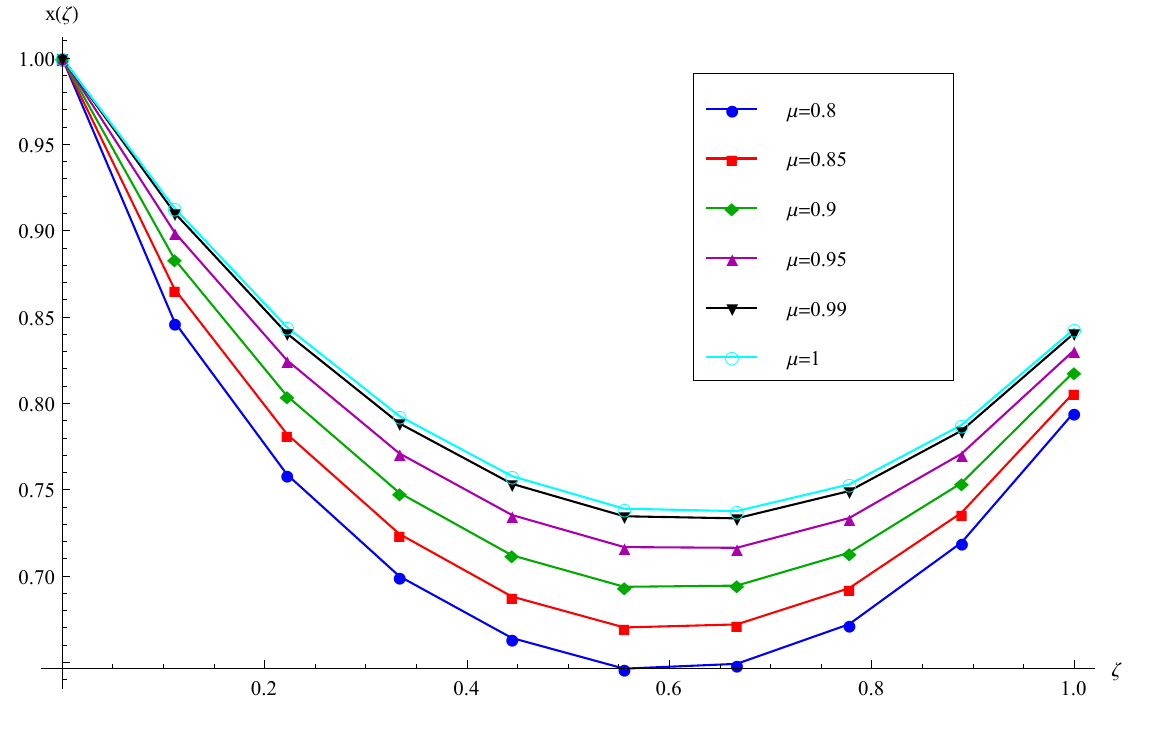}

Fig. 9: Using FTW  solutions of $u({\zeta})$ for $\mu=0.8, 0.85, 0.9, 0.95, 0.99$ and $ 1.$
\\

\includegraphics[scale=0.8]{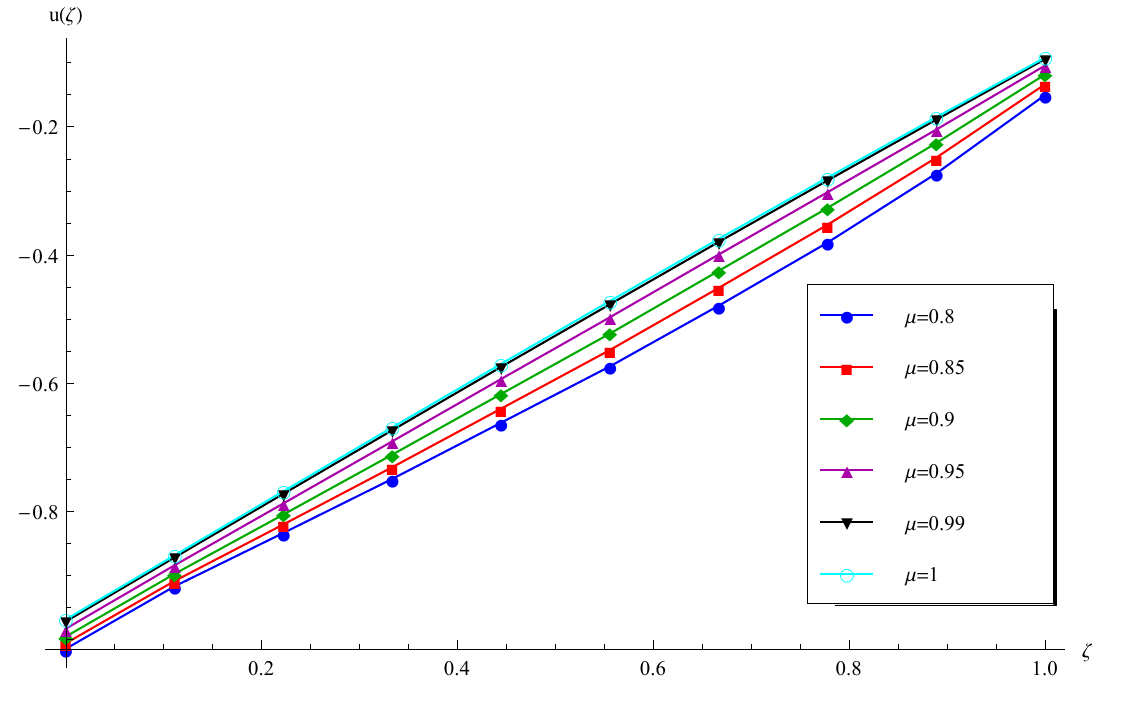}

Fig. 10:Using FTW  solutions of $u({\zeta})$ for $\mu=0.8, 0.85, 0.9, 0.95, 0.99$ and $1.$
\\
\\
\begin{landscape}
\begin{table}
\centering
\small\addtolength{\tabcolsep}{-1pt}
\caption{Approximate results of state function $x({\zeta})$ for 
$\mu=0.5, 0.75, 0.85, 0.95$ and $ 0.99.$ }
\begin{tabular}{|c|c|c|c|c|c|c|c|c|c|c|c|c|c|c|} \hline
${\zeta}$ &  \multicolumn{2}{|c|}{$\mu=0.5$} &\multicolumn{2}{|c|}{$\mu=0.75$} &\multicolumn{2}{|c|}{$\mu=0.85$}& \multicolumn{2}{|c|}{$\mu=0.95$} &\multicolumn{2}{|c|}{$\mu=0.99$}\\[0.5ex] \cline{2-11}
& TW &  FTW & TW & FTW & TW & FTW & TW & FTW & TW & FTW \\ \cline{2-11}
\hline
0.1  & 0.695033 & 0.706716 &  0.826483 & 0.826205 & 0.866578 & 0.866054 & 0.899198 &  0.898982 & 0.910304 & 0.91026\\
\hline
0.2 &  0.585076 & 0.622558  &  0.729052 & 0.734432 & 0.779992 & 0.782106 & 0.824404 & 0.824876 & 0.840254 & 0.840334\\
\hline
0.3  & 0.534201 & 0.55043 &  0.670491 & 0.675079 & 0.722854 &  0.724269 & 0.770655 & 0.770923 & 0.78828 & 0.788325\\
\hline
0.4  & 0.508309 & 0.535243 &  0.636575 & 0.633937 & 0.687537 &  0.688044 & 0.735357 & 0.735274  & 0.753365 &  0.753341\\
\hline
0.5  & 0.479621 & 0.520989 &  0.616584 & 0.622424 & 0.66812 & 0.670137 & 0.716283 & 0.71666 & 0.734461 & 0.734521\\
\hline
0.6 & 0.497024 &  0.528249 &  0.622844 &  0.62608 & 0.670801 & 0.671888 &  0.716051 & 0.716253  & 0.733272 & 0.733304\\
\hline
0.7  & 0.523824 & 0.567539 &  0.647981 & 0.651046 & 0.691985 & 0.692859 & 0.733308 & 0.733455 & 0.749074 & 0.749097\\
\hline
0.8 & 0.585355 & 0.644608  & 0.698397  & 0.702236 & 0.735577 & 0.73656 & 0.770585 & 0.770731 & 0.784045  &0.784067\\
\hline
0.9   & 0.70695 & 0.762689 &  0.780492  & 0.783741 & 0.805481 & 0.806219 & 0.830413 & 0.830503 & 0.840366 & 0.840378 \\
\hline
\end{tabular}
\label{table:5} 
\end{table}
\begin{table}
\centering
\small\addtolength{\tabcolsep}{-1pt}
\caption{Approximate results of control function $u({\zeta})$ for $\mu=0.5, 0.75, 0.85, 0.95$ and $ 0.99.$ }
\begin{tabular}{|c|c|c|c|c|c|c|c|c|c|c|c|c|c|c|} \hline
${\zeta}$ &  \multicolumn{2}{|c|}{$\mu=0.5$} &\multicolumn{2}{|c|}{$\mu=0.75$} &\multicolumn{2}{|c|}{$\mu=0.85$}& \multicolumn{2}{|c|}{$\mu=0.95$} &\multicolumn{2}{|c|}{$\mu=0.99$}\\[0.5ex] \cline{2-11}
& TW &  FTW & TW & FTW & TW & FTW & TW & FTW & TW & FTW \\ \cline{2-11}
\hline
0.1 &  -0.903678 &  -0.89907  & -0.92512 & -0.922047 &-0.909681 & -0.908465 & -0.884294 &  -0.883972 & -0.871842 & -0.871776  \\ 
\hline
0.2  & -0.86127  & -0.851525  & -0.847502 & -0.844525 &-0.821461  &-0.820291 & -0.787727 & -0.787566 & -0.772559  &-0.772542 \\
\hline
0.3 & -0.828241  &-0.828557  & -0.770249  &-0.765823 &-0.732867 & -0.731122  &-0.691084 & -0.690741 & -0.673518 & -0.673468 \\
\hline
0.4  &-0.791939  & -0.76738 &  -0.689978 &  -0.694032 & -0.642504 & -0.63987 & -0.59422 & -0.593637 & -0.574914 & -0.57481 \\
\hline
0.5 & -0.756231 & -0.720699  & -0.607542 & -0.599295 &-0.550974 & -0.548373  &-0.497573 & -0.49733 & -0.477135  &-0.477127 \\
\hline
0.6 &  -0.681542 & -0.660433 &  -0.511566 & -0.507927& -0.452795 & -0.451524  &-0.399688 & -0.399469 & -0.37991 & -0.379879 \\
\hline
0.7 & -0.615042 & -0.571855 &  -0.414964 & -0.41042 & -0.35438 & -0.352995 & -0.302644 & -0.30239&  -0.284058 & -0.284017 \\
\hline
0.8 &  -0.520397 & -0.446634 &  -0.30749 & -0.299428 & -0.250747 & -0.248397 & -0.204975 & -0.204612 & -0.189145 & -0.189092 \\
\hline
0.9 & -0.361272 &  -0.279856 &  -0.178901 & -0.169009 & -0.136917 & -0.133879  &-0.105213 & -0.104738 & -0.0947367 & -0.0946691 \\
\hline
\end{tabular}
\label{table:6} 
\end{table}
\end{landscape}

\begin{example}\label{3rd example}
  Consider the following time invariant problem:
    $$\min\Tilde{{J}}=\dfrac{1}{2}\int_{0}^{1}\left(\left(x({\zeta})-{\zeta}^{\mu}\right)^2+\left(u({\zeta})-{\zeta}^{\mu}-\Gamma{(\mu+1)}\right)^2\right)d{\zeta},$$
    subject to the system dynamics
    $$_0^{C}\mathcal{D}_{{\zeta}}^{\mu}x({\zeta})= -x({\zeta})+u({\zeta}),$$
    with initial condition
    $$x(0)=0.$$     
\end{example}

\includegraphics[scale=.85]{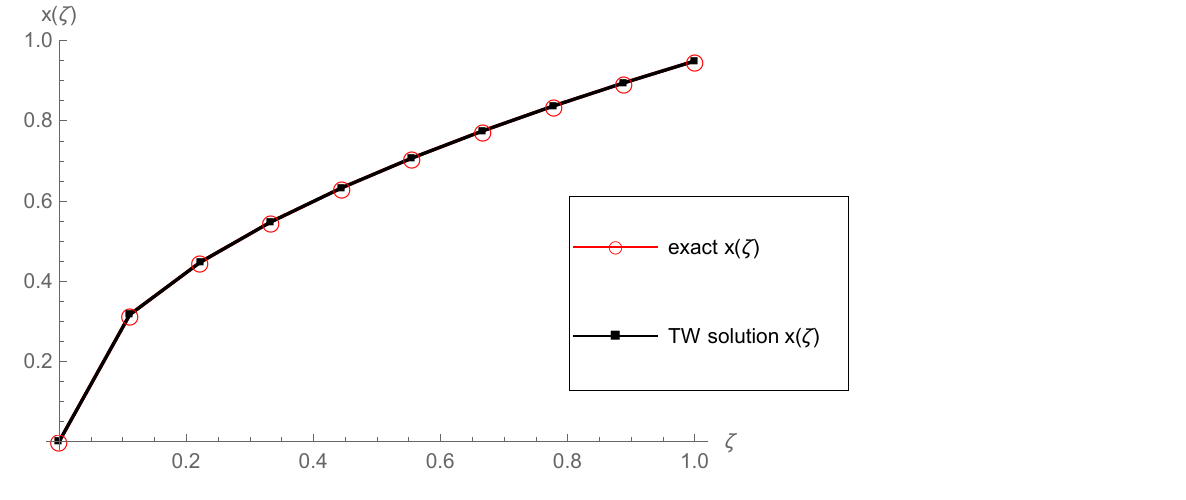}

Fig. 11: Exact and TW solution of state function $x({\zeta})$ for $\mu=0.5.$
\\\\
\includegraphics[scale=.85]{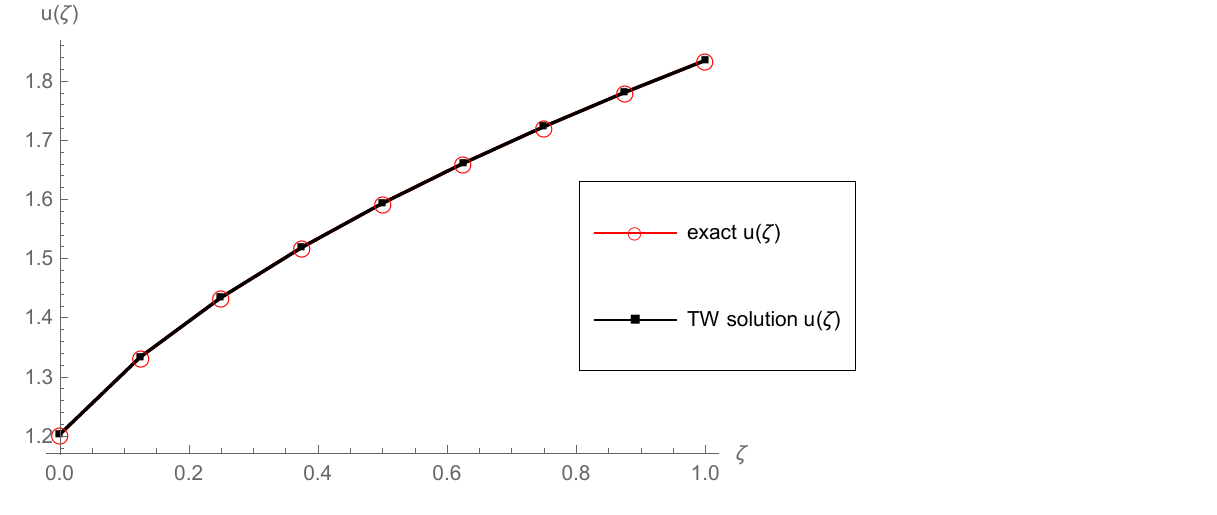}

Fig. 12: Exact and TW solution of control function $u({\zeta})$ for $\mu=0.5.$\\\\

The exact solution of example \ref{3rd example} is $(x({\zeta}), u({\zeta})) = \left({\zeta}^{\mu}, {\zeta}^{\mu}+\Gamma{(\mu +1)}\right)$. The above FOCP has been solved by using fractional Taylor wavelet and Taylor wavelet methods when ${M} = 4$ and ${{k}}=2$ for different values of $\mu$. Figs. 9, 11, and 13 show the approximate and exact solutions of the state function $x({\zeta})$ for $\mu = 0.5, 0.8 $, and $ 0.95$, respectively.  Similarly, Figs. 10, 12, and 14 show the approximate and exact solutions of the control function $u({\zeta})$ for $\mu = 0.5, 0.8,$ and $0.95$,  respectively. Figs. 11 and 12 have been obtained by the Taylor wavelet method and Figs. 13, 14, 15 and 16 have been obtained by the fractional Taylor wavelet method. We can easily observe that the approximate solution obtained by the proposed method is in very good agreement with the exact solution.\\\\

\includegraphics[scale=0.8]{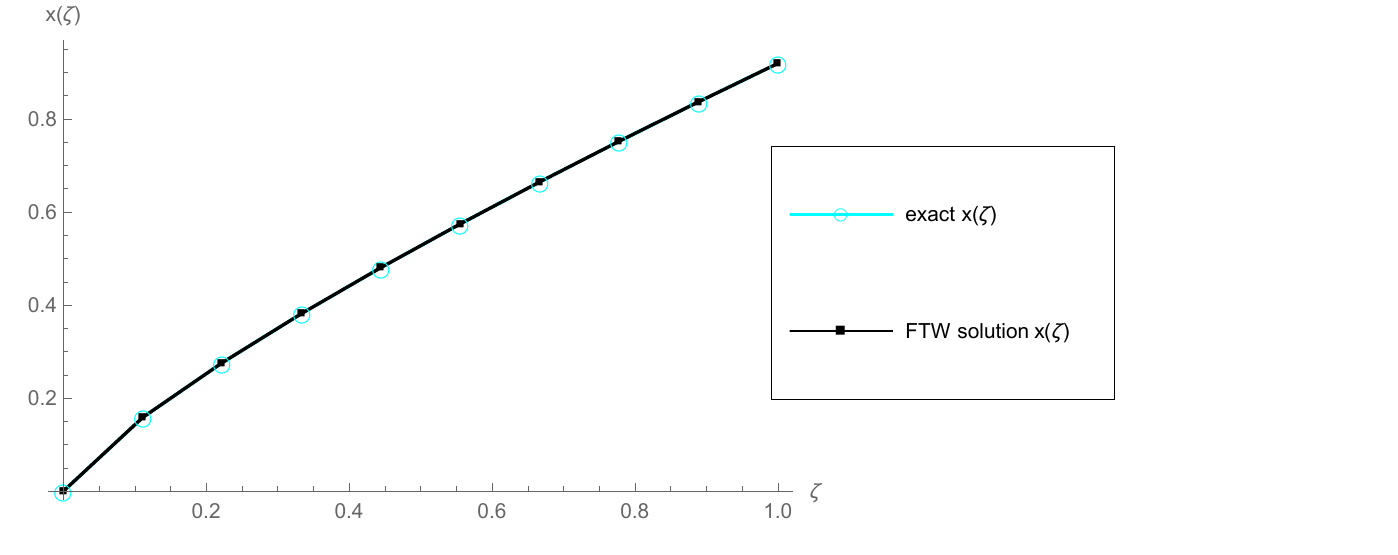}

Fig. 13: Exact and FTW solution of state function $x({\zeta})$ for $\mu=0.8.$
\\\\

\includegraphics[scale=.8]{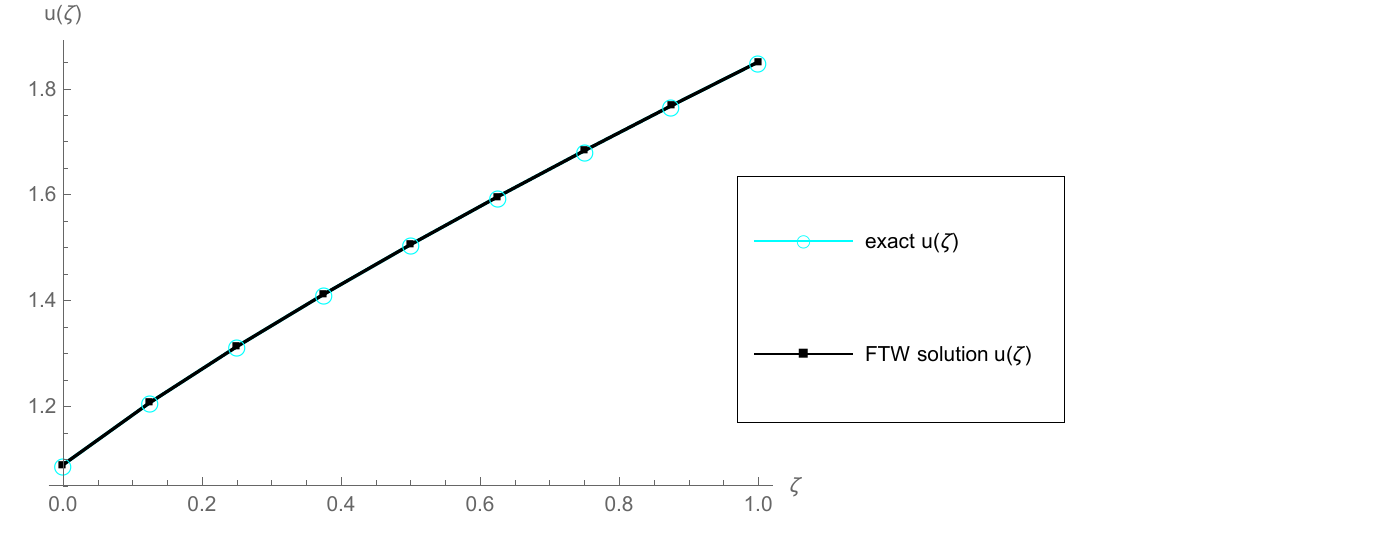}

Fig. 14: Exact and FTW solution of control function $u({\zeta})$ for $\mu=0.8.$\\\\

\includegraphics[scale=0.8]{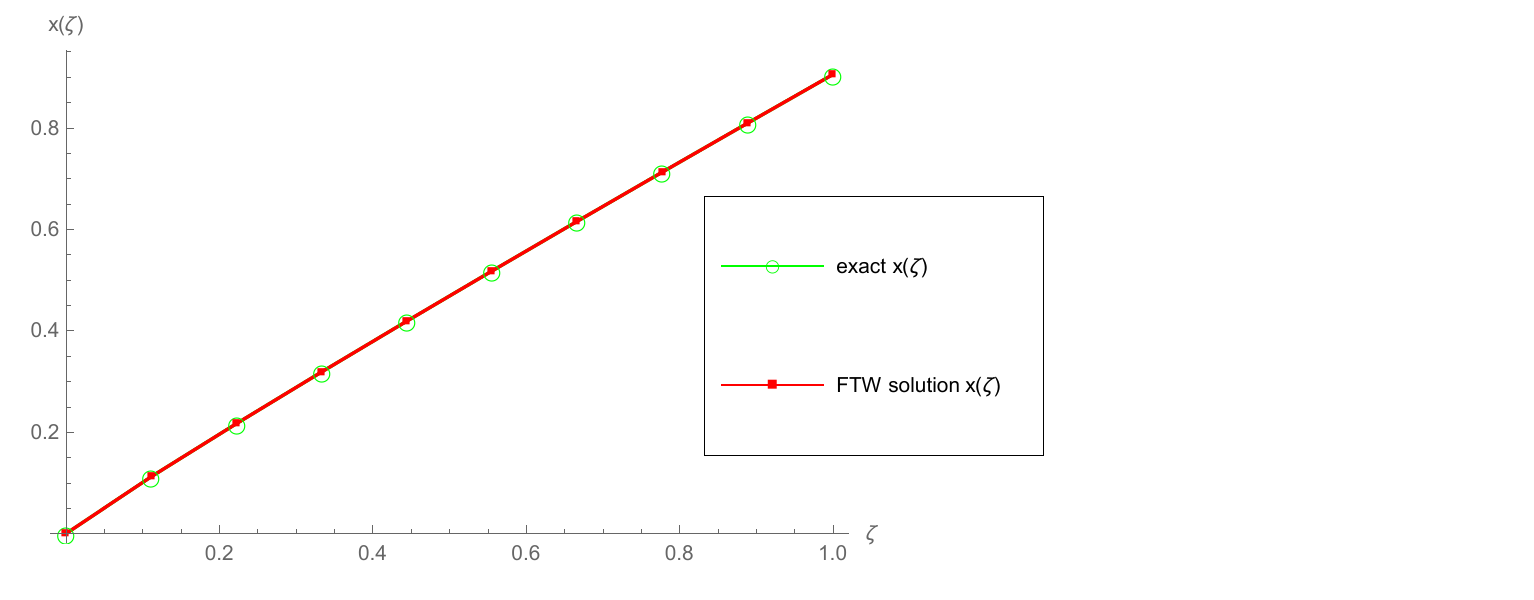}

Fig. 15: Exact and FTW solution of state function $x({\zeta}) $ for $\mu=0.95$.
\\

\includegraphics[scale=.8]{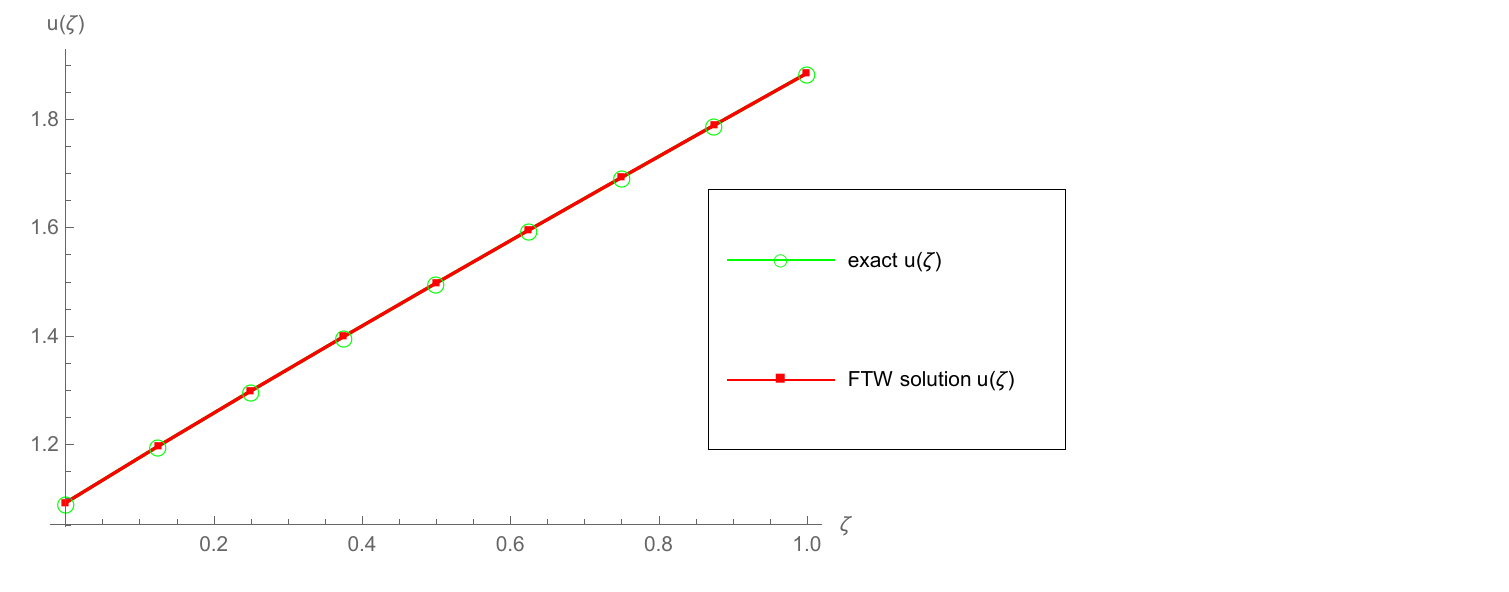}

Fig. 16: Exact and FTW solution of control function $u({\zeta}) $ for $\mu=0.95$. \\

\section{Conclusion}
The purpose of this paper is to describe a procedure for the numerical solution of fractional optimal control problems. In this work, fractional Taylor wavelets and Taylor wavelet methods have been used to solve fractional optimal control problems. As mentioned, the given proposed method transforms the fractional optimal problem into a system of equations that can be solved using the Lagrange multiplier technique. The numerical method based on fractional wavelets is a relatively new area of study. Fractional wavelets are piecewise functions that are also continuous functions with compact support $[0, 1]$.
   The novelty of this paper is the use of a fractional Taylor wavelet to solve fractional optimal control problems using  operational matrices. The implementation of the proposed method is very simple and effective for solving fractional optimal control problems. So, the experimental results from the test examples show that the proposed numerical techniques are accurate with high accuracy. Moreover, error estimation and convergence analysis of the proposed numerical technique have been also established. \\
\\
\textbf{Acknowledgement}
The ``University Grants Commission (UGC)" fellowship scheme, which provided financial support with NTA Ref. No. 201610127052 is gratefully acknowledged by the second author. \\
\textbf{Conflict of interest } The authors declare that they have no conflict of interest. \\

\end{document}